\documentclass[twoside,10pt]{article}
\usepackage{mathrsfs}
\usepackage{amsfonts}
\usepackage{amssymb}  %,showkeys
\usepackage{amsmath,amscd,graphicx}
\usepackage{paralist}  % generate "compactenum"
\usepackage{amsthm}  % generate "begin proof"
\makeatletter

\newcommand{\Rmnum}[1]{\expandafter\@slowromancap\romannumeral #1@}
\makeatother

\def\D{\mathcal{D}}
\def\F{\mathcal{F}}
\def\R{\mathbb{R}}
\def\T{\mathbb{T}}
\def\E{\mathcal{E}}
\def\G{\mathcal{G}}
\def\H2{H^2(\R^N)}
\def\L2{L^2(\R^N)}

\def\p{\partial}
\def\n{\nabla}

\def\nn{\n_*}
\def\na{\n_{\mathcal{A}}}
\def\nna{\n_{\mathcal{A}*}}
\def\a{\mathcal{A}}

       \newtheorem{lemma}{\bf Lemma}[section]
       \newtheorem{theorem}{\bf Theorem}[section]
       \newtheorem{proposition}{\bf Proposition}[section]
       
       \newtheorem{definition}{\bf Definition}[section]
       \newtheorem{remark}{\bf Remark}[section]

       \numberwithin{equation}{section}

\begin{document}

\title{{\Large \textbf{Global existence and large time behavior for primitive equations with free boundary}}
\footnotetext{\small *Corresponding author.}
 \footnotetext{\small E-mail address: liangcc@cqu.edu.cn} 
 \footnotetext{\small Accepted for publication in \textbf{SCIENCE CHINA Mathematics}.}}

\author{{Hai-Liang Li$^{1,2}$ and Chuangchuang Liang$^3$$^\ast$}\\[2mm]
\normalsize\it $^1$School of Mathematical Sciences, Capital Normal University,\\
\normalsize\it   Beijing 100048, PR China\\
\normalsize\it $^2$Academy for Multidisciplinary Studies, Capital Normal University,\\
\normalsize\it   Beijing 100048, PR China\\
\normalsize\it $^3$College of Mathematics and Statistics, Chongqing University,\\
\normalsize\it   Chongqing 400044, PR China
}

\date{}

\maketitle

\begin{quote}
\small \textbf{Abstract}: In the present paper, the primitive equations, which can be used to simulate the large scale motion of ocean and atmosphere, are considered in the three-dimensional domain bounded below by a fixed solid boundary and above by a free moving boundary. The global existence and uniqueness of strong solutions are established and the long time convergence to the equilibrium state is showed either at exponential rate for horizontal periodic domain or at algebraic rate for horizontal whole space.


\indent \textbf{AMS (2010) Subject Classification}: 35Q30, 35R35, 35A01, 76E20, 35Q35

\end{quote}

%35Q30 Navier-Stokes equations;
%35Q35  	PDEs in connection with fluid mechanics
%35A01  	Existence problems: global existence, local existence, non-existence
%76E20  	Stability and instability of geophysical and astrophysical flows
%35R35  	Free boundary problems

%%%%%%%%%%%%%%%%%%%%%%%%%%%%%  Introduction %%%%%%%%%%%%%%%%%  %%%%%%%%%%%% %%%%%%%%%%%%%%%%%%%%%%%%%%%%%%%%%%%%%%

\section{Introduction}
\par The primitive equations for the large scale dynamics of ocean and atmosphere, which were introduced by Richardson \cite{Ric} in 1922 and were applied to model the atmosphere by Smagorinsky \cite{Sma} and the ocean circulation by Bryan \cite{Bryan}, can be derived from the Navier-Stokes equations under the Boussinesq and hydrostatic approximations, refer to \cite{GH, Majda, Ped, Sal, Sma, TZ, zeng} and the reference therein. Moreover, since the vertical scale motion in the ocean and atmosphere is much smaller than the horizontal one ($10–20$ km versus several thousands of kilometers), the natural simplification of the model for the motion of ocean and atmosphere leads to the primitive equations by the so-called hydrostatic approximation.
\par The primitive equations without the influences of the thermodynamics and the salinity are given by
\begin{equation}\label{GF1}
\left\{\begin{aligned}
&\p_tv+v\cdot\nn v+w\p_3v-\Delta v+\nn P+f\vec{\kappa}\times v=0,\\
&\p_3P=-g,\\
&\nn\cdot v+\p_3w=0,
\end{aligned}\right.\end{equation}
where the vector $ v=(v^1,v^2)^T$ denotes the horizontal velocity, the scalar $ w $ is the vertical velocity, $ P $ is the pressure, \begin{equation}\label{1111}
\nn \phi:=\left(\frac{\p\phi}{\p x_1},\frac{\p\phi}{\p x_2}\right)^T=(\p_1\phi,\p_2\phi)^T,\quad \p_3\phi:=\frac{\p\phi}{\p{x_3}},\quad \Delta\phi=\p_1^2\phi+\p_2^2\phi+\p_3^2\phi
\end{equation}
for any function $\phi$
and
$$ \vec{\kappa}:=(0,0,1)^T, \qquad \vec{\kappa}\times v=(-v^2,v^1).$$
The positive constants $f$ and $g$ are coefficients of the Coriolis force and the gravity, respectively. Without loss of generality, the effects of the thermodynamics and salinity are omitted in this paper for simplicity.

\par The mathematical analysis on the primitive equations goes back to Lions, Temam and Wang \cite{LTW1, LTW2, LTW3}, where the global existence and the attractors of weak solutions were established in two and three dimensional domain. And the uniqueness of the weak solutions in some suitable spaces was proved in \cite{B,PTZ,T,KPRZ,LT}. The local well-posedness of the strong solutions to the primitive equations was investigated in \cite{GG}. The global existence of strong solutions to the primitive equations with full viscosity and diffusivity was showed in \cite{BKL} for the two-dimensional case, while the global strong solutions for the three-dimensional case were established in \cite{CT} and \cite{K} by using a different approach. An important observation in \cite{CT} is noteworthy that the unkown pressure is in fact a two-dimensional function (a function of the horizontal variables and time) and then the a-priori estimates were obtained by integrating the horizontal momentum equations over the vertical direction from the bottom to the top, which will be used later to improve the regularity of the free surface in our present paper. Recently, Cao, Li and Titi \cite{CLT1,CLT2,CLT3,CLT4,CLT5,CLT6} have made some progresses on the intial boundary value problem to the primitive equations with partial viscosity or partial heat diffusion, where the local and global well-posedness of the solutions to the cases with partial viscosity or partial diffusion was established. However, for the inviscid primitive equations with or without coupling to the temperature, it is showed in \cite{CIT,W} that the smooth solutions would blow up in finite time. The rigorous mathematical justification of the small aspect ratio from the Navier-Stokes equtions to the primitive equations, i.e., hydrostatic approximation, was studied in \cite{AG}, where the weak convergences were established. The strong convergences, which are global and uniform in time, and the convergence rate were established in \cite{LT2} and \cite{F} by different ways.

\par The systems considered in all the above are assumed to hold in the fixed domain, i.e., the domain is independent of the time. Both physically and mathematically, it is also important and necessary to study the free boundary value problem to the primitive equations. Crowley \cite{Cro} firstly treated the boundary as a free surface and considered the numerical evaluation. The coupled atmosphere and ocean model with the free interface was derived in \cite{LTWF} and the local existence for the inviscid case in the real analytic space was established in \cite{IKZ}. The free boundary value problem of primitive equations with the effects of the viscosity, thermodynamics and salinity was studied in \cite{HT1, HT2}, where the local well-posedness in the Sobolev-Sobodetski\`{\i} spaces was showed. However, there are not any results on the global existence and larger time behavior of the solutions to the free boundary value problem of the primitive equations, which is our aim in this paper.\\

\par In the present paper, we consider the global well-posedness and large time behavior of the strong solutions to the free boundary value problem of primitive equations \eqref{GF1} in the horizontal periodic domain or horizontal infinite domain
\begin{equation}\label{D11}
\Omega_t:=\{(x',x_3)|x' \in \Gamma,-b<x_3<\zeta(t,x'),\ b>0\}
\end{equation}
for the horizontal spatial domain $\Gamma:=\mathbb{T}^2\text{ or }\mathbb{R}^2$.
The conditions on the free surface will be derived from the free boundary value problem to the incompressible Navier-Stokes equations by the hydrostatic approximation.

\subsection{The derivation of the free boundary value problem}
\par The free boundary value problem for \eqref{GF1} can be derivated from the free boundary value problem of the incompressible Navier-Stokes equations under the hydrodynamic approximation (refer to \cite{TZ} and the references therein for the detail). Indeed, the incompressible Navier-Stokes equations with free surface are described as
\begin{equation}\label{DD1}
\left\{\begin{aligned}
&\p_tU_{\varepsilon}+U_{\varepsilon}\cdot\n_{y} U_{\varepsilon}-\n_{y}\cdot\mathcal{T}_{\varepsilon}+f\vec{\kappa}\times U_{\varepsilon}=-\frac{g}{\varepsilon}\vec{e}_3,\\
&\n_{y}\cdot U_{\varepsilon}=0
\end{aligned}\right.
\end{equation}
in the horizontal periodic domain or horizontal infinite domain $$\{(y',y_3)|y'\in \Gamma,-\varepsilon b<y_3<\varepsilon\eta_{\varepsilon}(t,y'),\ b>0,\varepsilon>0\}$$
with the following boundary conditions
\begin{equation}\label{DD2}
\left\{\begin{aligned}
&\mathcal{T}_{\varepsilon}\vec{\nu}_{\varepsilon}=-P_0\vec{\nu}_{\varepsilon} \quad\text{and}\quad\varepsilon\p_t\eta_{\varepsilon}+\varepsilon U_{\varepsilon}^*\cdot\nabla_{y*}\eta_{\varepsilon} =U^3_{\varepsilon}\qquad &\text{for }y_3=\varepsilon\eta_{\varepsilon}(t,y'),\\
&U_{\varepsilon}=0\qquad &\text{for }y_3=-\varepsilon b,
\end{aligned}\right.
\end{equation}
where $U_{\varepsilon}:=(U_{\varepsilon}^1,U_{\varepsilon}^2,U_{\varepsilon}^3)^T$ and $P_{\varepsilon}$ denote the velocity and the pressure of the flow respectively, $U_{\varepsilon}^*:=(U_{\varepsilon}^1,U_{\varepsilon}^2)^T$ is the horizontal velocity, the spatial derivatives $\nabla_{y}$ are given by $\nabla_{y}\phi=(\p_{y_1}\phi,\p_{y_2}\phi,\p_{y_3}\phi)^T$ and $\nabla_{y*}\phi=(\p_{y_1}\phi,\p_{y_2}\phi)^T$ for any function $\phi$, the stress tensor $\mathcal{T}_{\varepsilon}$ takes the form
$$\mathcal{T}_{\varepsilon}=-P_{\varepsilon}I+\left(\begin{matrix}
&\p_{y_1}U_{\varepsilon}^1&\p_{y_2}U_{\varepsilon}^1&\varepsilon^2\p_{y_3}U_{\varepsilon}^1\\
&\p_{y_1}U_{\varepsilon}^2&\p_{y_2}U_{\varepsilon}^2&\varepsilon^2\p_{y_3}U_{\varepsilon}^2\\
&\p_{y_1}U_{\varepsilon}^3&\p_{y_2}U_{\varepsilon}^3&\varepsilon^2\p_{y_3}U_{\varepsilon}^3
\end{matrix}\right),$$
$\frac{g}{\varepsilon}$ denotes the positive coefficient of the gravity, the positive constant $P_0$ means the pressure of the atmosphere, the hrozontal spatial domain $\Gamma:=\mathbb{T}^2 \text{ or }\mathbb{R}^2$ and the unit outward normal vector $\vec{\nu}_\varepsilon$ reads
$$\vec{\nu}_{\varepsilon}:=\frac{\left(-\varepsilon\p_{y_1}\eta_{\varepsilon},-\varepsilon\p_{y_2}\eta_{\varepsilon},1\right)^T}{\sqrt{1+\varepsilon^2|\p_{y_1}\eta_{\varepsilon}|^2+\varepsilon^2|\p_{y_2}\eta_{\varepsilon}|^2}}.$$

\par Define the scaling transform $\tilde{\varPsi}_t$ by
\begin{equation}\label{GG1}
\begin{aligned}
\tilde{\varPsi}_t:\left\{(x',x_3)|x'\in \Gamma,-b<x_3<\zeta_{\varepsilon}(x',t)\right\}&\longrightarrow\{(y',y_3)|y'\in \Gamma,-\varepsilon b<y_3<\varepsilon\eta_{\varepsilon}(t,y')\}\\
(x',x_3)&\mapsto(y',y_3)=(x',\varepsilon x_3)
\end{aligned}\end{equation}
and rescale the velocity, pressure and the free surface to
\begin{equation}\label{GG2}
\left\{\begin{aligned}
&v_{\varepsilon}(t,x',x_3)=U_{\varepsilon}^*\left(t,\tilde{\varPsi}_t(x',x_3)\right),\quad w_{\varepsilon}(t,x',x_3)=\frac{U_{\varepsilon}^3\left(t,\tilde{\varPsi}_t(x',x_3)\right)}{\varepsilon},\\
&\tilde{P}_{\varepsilon}(t,x',x_3)=P_{\varepsilon}\left(t,\tilde{\varPsi}_t(x',x_3)\right)\quad\text{and}\quad  \zeta_{\varepsilon}(t,x')=\eta_\varepsilon(t,x').
\end{aligned}\right.
\end{equation}

\par By \eqref{GG1}, we obtain the equations for \eqref{GG2} from \eqref{DD1}-\eqref{DD2} as
\begin{equation}\label{KL1}
\left\{\begin{aligned}
&\p_tv_{\varepsilon}+v_{\varepsilon}\cdot\nn v_{\varepsilon}+w_{\varepsilon}\p_3v_{\varepsilon}+\nn \tilde{P}_{\varepsilon}-\Delta v_{\varepsilon}+f\vec{\kappa}\times v_{\varepsilon}=0,\\
&\varepsilon^2\left[\p_tw_{\varepsilon}+v_{\varepsilon}\cdot\nn w_{\varepsilon}+w_{\varepsilon}\cdot\p_3w_{\varepsilon}-\Delta w_{\varepsilon}\right]+\p_3\tilde{P}_{\varepsilon}=-g,\\
&\nn\cdot v_{\varepsilon}+\p_3w_{\varepsilon}=0
\end{aligned}\right.
\end{equation}
for $(x',x_3)\in\left\{(x',x_3)|x'\in \Gamma,-b<x_3<\zeta_{\varepsilon}(t,x')\right\}$ and the boundary conditions
\begin{equation}\label{KL2}
\left\{\begin{aligned}
&\tilde{P}_{\varepsilon}\nn\zeta_{\varepsilon}+\vec{n}_{\varepsilon}\cdot\n v_{\varepsilon}=P_{0}\nn\zeta_{\varepsilon} \qquad &\text{for }x_3=\zeta_{\varepsilon}(t,x'),\\
&-\tilde{P}_{\varepsilon}+\varepsilon^2\vec{n}_{\varepsilon}\cdot\n w_{\varepsilon}=-P_{0} \qquad &\text{for }x_3=\zeta_{\varepsilon}(t,x'),\\
&\p_t\zeta_{\varepsilon}+v_{\varepsilon}\cdot\nn\zeta_{\varepsilon} =w_{\varepsilon}\qquad &\text{for }x_3=\zeta_{\varepsilon}(t,x'),\\
&v_{\varepsilon}=w_{\varepsilon}=0\qquad &\text{for }x_3=-b,
\end{aligned}\right.
\end{equation}
where the outward normal vector $\vec{n}_\varepsilon$ is
$$\vec{n}_{\varepsilon}:=(-\p_1\zeta_{\varepsilon},-\p_2\zeta_{\varepsilon},1)^T$$
and $\n=(\nn,\p_3)^T$, $\nn$ and $\Delta$ are the spatial derivatives defined in \eqref{1111}.

\par Assuming that the convergences hold as $\varepsilon$ tends to zero, i.e.,
\begin{equation*}
\begin{aligned}
\left(v_{\varepsilon},w_{\varepsilon},\tilde{P}_{\varepsilon},\zeta_{\varepsilon}\right)\rightarrow (v,w,\tilde{P},\zeta),
\end{aligned}
\end{equation*}
we formally obtain the free boundary value problem for $(v,w,\tilde{P},\zeta)$ from \eqref{KL1}-\eqref{KL2}
\begin{equation}\label{p1}
\left\{\begin{aligned}
&\p_tv+v\cdot\nn v+w\p_3v-\Delta v+\nn P+f\vec{\kappa}\times v=0,\\
&\p_3P=0,\\
&\nn\cdot v+\p_3w=0
\end{aligned}\right.\end{equation}
for $t>0$ and $x\in\Omega_t$ defined in \eqref{D11}, and the following boundary conditions
\begin{equation}\label{p2}\left\{
\begin{aligned}
&P=P_0+g\zeta\quad\text{and} \quad \vec{n}\cdot\n v=0 \qquad &\text{on }\Gamma_t,\\
&\p_t\zeta+v\cdot\nn\zeta=w\qquad&\text{on }\Gamma_t,\\
&v=w=0 \qquad&\text{on }\Sigma_b,
\end{aligned}\right.
\end{equation}
where the pressure $P$ and the outward normal vector $\vec{n}$ are defined by
$$P(t,x',x_3)=\tilde{P}(t,x',x_3)+gx_3\quad\text{and}\quad \vec{n}:=(-\p_1\zeta,-\p_2\zeta,1)^T,$$
and the free surface $\Gamma_t$ is given by
$$\Gamma_t:=\{(x',x_3)|x'\in \Gamma,x_3=\zeta(t,x')\}$$
with the horizontal spatial domain $\Gamma=\mathbb{T}^2\text{ or } \mathbb{R}^2$.
The initial data are given by
\begin{equation}\label{IIII}
v(0,x',x_3)=v_0(x',x_3) \text{ and }\zeta(0,x')=\zeta_0(x')
\end{equation}
for $x'\in \Gamma$ and $(x',x_3)\in\Omega_0$ defined by
\begin{equation}\label{ID}
\Omega_0:=\{(x',x_3)|x'\in \Gamma,-b<x_3<\zeta_0(x')\}.
\end{equation}

\par By comparing the primitive equations with free boundary, there are lots of works on the free boundary value problem to the Navier-Stokes equations. Here we mainly introduce some well-posedness results. The local well-posedness in Sobolev space to the viscous incompressible flow with the free boundary was showed in \cite{Beale2,wu}, and the global existence and large time behavior of the solution, perturbed around the constant stationary state, were studied in \cite{Hataya}. The motion of a viscous compressible barotropic fluid in $\R^3$ bounded by free surface with or without surface tension was investigated in \cite{WZ1,WZ2}. Solonnikov also did many works on the free boundary value problem of the compressible or incompressible Navier-Stokes equations, refer to \cite{Sol} and the reference therein. Recently, the free boundary value problem to the incompressible Navier-Stokes equations without surface tension in the horizontal periodic or horizontal infinite domain has been studied in \cite{GuoY1,GuoY2,GuoY3}, where the global well-posedness of the solutions was established and the long time convergence to the equilibrium state was showed either at almost exponential rate for horizontal periodic domain or at algebraic rate for horizontal whole space.

\subsection{Main results}
\par Based on the incompressible condition $\eqref{p1}_3$ and the kinematic boundary conditions $\eqref{p2}$, we can have
\begin{equation*}
\frac{d}{dt}\int_{\Gamma}\zeta(t,x')dx'=\int_{\Gamma}[-v\cdot\nn\zeta+w]\left(t,x',\zeta(t,x')\right)dx'=\int_{\Omega_t}[\nn\cdot v+\p_3w](t,x',x_3)dx=0,
\end{equation*}
where $\Gamma$ denotes either the periodic domain $\T^2$ or the whole space $\R^2$. Therefore, without loss of generality, we assume the initial datum $\zeta_0$ of $\zeta$ satisfies
\begin{equation}\label{fbass}
\int_{\Gamma}\zeta_0(x')dx'=0,
\end{equation}
so as to have
\begin{equation}\label{fbass1}
\int_{\Gamma}\zeta(t,x')dx'=\int_{\Gamma}\zeta_0(x')dx'=0.
\end{equation}

\par In the present paper, we consider the global well-posedness and large time behavior of the strong solutions to the free boundary value problem \eqref{p1}-\eqref{IIII} for primitive equations near the constant steady state
\begin{equation}\label{sc}
(\bar{v},\bar{w},\bar{\zeta})=(0,0,0)\ \text{ and }\ \bar{P}=P_0.
\end{equation}

\par To overcome the difficulties which are caused by the free surface, we will use the harmonic extension introduced by Beale \cite{Beale2} to flatten the free boundary.

\par In view of \eqref{fbass1}, we introduce the following fixed domain $\Omega$ by
$$\Omega:=\{(x',x_3)|x'\in \Gamma,-b<x_3<0\}$$
and define the flatting transform $\Psi_t$ from $\Omega$ to $\Omega_t$, such that
\begin{equation}\label{P1}
\Psi_t:(x',x_3)\in \Omega\longmapsto \Omega_t\ni(x',x_3+\theta),
\end{equation}
where the function $\theta$ is $$\theta(t,x',x_3):=\chi(x_3)\tilde{\zeta}(t,x',x_3),\qquad \text{for }(x',x_3)\in\Omega.$$
Note that $\tilde{\zeta}$ is the harmonic extention of $\zeta$ given by
\begin{equation*}
\tilde{\zeta}(t,x',x_3):=\left\{
\begin{aligned}
&\sum_{n\in \mathbb{Z}^2}e^{2i\pi n\cdot x'}e^{|n|x_3}\hat{\zeta}(t,n)\qquad\text{for } x'\in \mathbb{T}^2,\\
&\int_{\R^2}e^{2i\pi\xi\cdot x'}e^{|\xi|x_3}\hat{\zeta}(t,\xi)d\xi\qquad\text{for }x' \in \mathbb{R}^2
\end{aligned}\right.
\end{equation*}
and $\chi(x_3)$ is a cutoff function, satisfying
$$\left\{\begin{aligned}
&\chi(x_3)\in C^{\infty}_0((-b,0])\quad 0\leq \chi(x_3)\leq1,\\
&\chi(x_3)=0 \qquad \text{for }x_3\in (-b,-\frac{3b}{4}),\\
&\chi(x_3)=1\qquad \text{for }x_3\in (-\frac{b}{4},0],\\
\end{aligned}\right.$$
where $\hat{\zeta}$ denotes the Fourier transform of $\zeta$. By the definition \eqref{P1}, we know that the map $\Psi_t$ transforms the upper boundary $\Gamma$ of $\Omega$ into the free surface $\Gamma_t$ and keeps the bottom $\Sigma_{b}$ of $\Omega$ invariant.
\par Then the free boundary value problem \eqref{p1}-\eqref{IIII} is reformulated into the following initial boundary value problem
\begin{equation}\label{p3}
\left\{\begin{aligned}
&\p_tv-\p_t\theta K\p_{3}v+v\cdot\nna v+wK\p_3v-\Delta_{\mathcal{A}} v+\nna P+f\vec{\kappa}\times v=0,\\
&K\p_3P=0,\\
&\nna\cdot v+K\p_3w=0
\end{aligned}\right.\end{equation}
for $t>0$ and $x\in\Omega$ with the kinematic boundary conditions
\begin{equation}\label{p4}\left\{
\begin{aligned}
&P=P_0+g\zeta\quad\text{and}\quad  \vec{n}\cdot\na v=0 \qquad &\text{on }\Gamma,\\
&\p_t\zeta+v\cdot\nn\zeta=w\qquad&\text{on }\Gamma,\\
&v=w=0 \qquad&\text{on }\Sigma_b,
\end{aligned}\right.\end{equation}
and the initial condition
\begin{equation}\label{IC}
(v(t,x',x_3),\zeta(t,x'))|_{t=0}=\left(v_0\left(\Psi_0(x',x_3)\right),\zeta_0(x')\right)
\end{equation}
for $(x',x_3)\in\Omega$ and $\Psi_0=\Psi_t|_{t=0}$,
where we have
$$\mathcal{A}:=\left(\begin{matrix}
1&0&-AK\\
0&1&-BK\\
0&0&K
\end{matrix}\right)$$
and $A:=\p_1\theta$, $B:=\p_2\theta$ and $K=J^{-1}:=(1+\p_3\theta)^{-1}$, $$\na\phi:=\a\n\phi:=(\bar{\p_1}\phi,\bar{\p_2}\phi,\bar{\p_3}\phi)^T=(\p_1\phi-AK\p_3\phi,\p_2\phi-BK\p_3\phi,K\p_3\phi)^T$$ and $\nna\phi:=(\bar{\p_1}\phi,\bar{\p_2}\phi)^T$ for any function $\phi$. Note that we use the same symbol $\Gamma$ to denote the uper boundary of $\Omega$, i.e.
$$\Gamma:=\{(x',0)|x'\in\mathbb{T}^2\text{ or }x'\in\mathbb{R}^2\}.$$

\par Based on the equations $\eqref{p3}_{2,3}$ and the kinematic boundary conditions $\eqref{p4}_{1,3}$, we get
\begin{equation}\label{p5}
	P(t,x',x_3)=P_0+g\zeta(t,x')\quad \text{and}\quad
	w(t,x',x_3)=-\int_{-b}^{x_3}(J\nna\cdot v)(t,x',\tau)d\tau.
\end{equation}
By \eqref{p5}, we can rewrite the equations \eqref{p3} and \eqref{p4} as
\begin{equation}\label{p6}
\left\{\begin{aligned}
&\p_tv-\p_t\theta K\p_{3}v+v\cdot\nna v+wK\p_3v-\Delta_{\mathcal{A}} v+g\nn\zeta +f\vec{\kappa}\times v=0,\\
&\nna\cdot v+K\p_3w=0
\end{aligned}\right.
\end{equation}
in the fixed domain $\Omega$ with the kinematic boundary conditions
\begin{equation}\label{p61}
\left\{\begin{aligned}
&\vec{n}\cdot\na v=0\qquad &\text{on } \Gamma,\\
&\p_t\zeta+v\cdot\nn\zeta=w\qquad&\text{on }\Gamma,\\
&v=w=0 \qquad&\text{on }\Sigma_b.
\end{aligned}\right.
\end{equation}

\par To obtain the strong solution and establish the regularities to the free boundary value problem \eqref{p3}-\eqref{IC}, the compatible conditions in the horizontal periodic domain or horizontal infinite domain are required, i.e.,
\begin{equation}\label{compatibility}
\left\{\begin{aligned}
&\p_t^jv(0,\cdot)=0\qquad &\text{on }\Sigma_{b},\\
&\p_t^j\left(\vec{n}\cdot\na v\right)(0,\cdot)=0\qquad &\text{on }\Gamma
\end{aligned}\right.
\end{equation}
for $j=0,1,2$.

\par For any strong solution $(v,\zeta)$ to the free boundary value problem \eqref{p3}-\eqref{IC}, we define the energy $\E(t)$ and the dissipation $\D(t)$ as
\begin{equation}\label{DED}
\begin{aligned}
\E(t):=&\sum_{i=0}^{2}\left(\|\p_t^iv(t,\cdot)\|_{4-2i}+|\p_t^i\zeta(t,\cdot)|_{4-2i}\right),\\
\D(t):=&\sum_{i=0}^{2}\|\p_t^iv(t,\cdot)\|_{5-2i}+|\nn\zeta(t,\cdot)|_{3}+\sum_{i=1}^{3}|\p_t^i\zeta(t,\cdot)|_{\frac{7}{2}-2(i-1)},\\
\F(t):=&|\nn\zeta(t,\cdot)|_{\frac{7}{2}},\\
\G(T):=&\sup_{0\leq t\leq T}(\E^2(t)+\F^2(t))+\int_{0}^{T}\D^2(t)dt,
\end{aligned}\end{equation}
where $\|\cdot\|_s$ or $|\cdot|_r$ denotes the norms of Sobolev space $H^s(\Omega)$ or $H^r(\Gamma)$, respectively and if $s=0$ or $r=0$, it means the norm of $L^2(\Omega)$ or $L^2(\Gamma)$.

\par For the free boundary value problem \eqref{p3}-\eqref{IC} in the horizontal periodic domain, we have the following results on the global existence and long time behavior of the strong solutions.

\begin{theorem}(\textbf{Horizontal periodic domain.})\label{th1}
Assume that the initial data $(v_0,\zeta_0)\in H^4(\Omega_0)\times H^{\frac{9}{2}}(\mathbb{T}^2)$ and \eqref{fbass} and \eqref{compatibility} hold. There exists a small constant $\delta_0>0$, such that if the initial data satisfy
$$\E^2(0)+\F^2(0)\leq \delta_0,$$
then the free boundary value problem \eqref{p3}-\eqref{IC} has a unique strong global solution
$$(v,w,\zeta)\in L^\infty\left([0,\infty),H^4(\Omega)\times H^3(\Omega)\times H^{\frac{9}{2}}(\mathbb{T}^2)\right)\cap L^2\left([0,\infty),H^5(\Omega)\times H^4(\Omega)\times H^{\frac{9}{2}}(\mathbb{T}^2)\right),$$
satisfying
$$\G(t)\leq C_1(\E^2(0)+\F^2(0)),\qquad  t>0$$
for some positive constant $C_1$ independent of $\delta_0$ and the time $t$.\\
Furthermore, the solution converges exponentially to the constant steady state $(\bar{v},\bar{w},\bar{\zeta},\bar{P})=(0,0,0,P_0)$, i.e., there exists a positive constant $\gamma_0$, such that
$$\E(t)\leq \E(0) e^{-\gamma_0t},\qquad  t>0.$$
\end{theorem}

\begin{remark}
It should be pointed out that Guo and Tice \cite{GuoY3} considered the incompressible Navier-Stokes equations with the free surface in the horizontal periodic domain and established the global existence of the strong solution and the long time behavior to the constant equilibrium at almost exponential decay rate. However, in our case, we obtain the solution to the free boundary value problem \eqref{p3}-\eqref{IC} exponentially decays to the constant steady state. The reason is that due to the hydrostatic approximation, the vertical momentum equation is simplified to $\eqref{p3}_2$, which, by the boundary condition \eqref{p4}, immediately gives the new relation between the pressure $P$ and the free boundary $\zeta$ as
	$$P(t,x',x_3)=P_0+g\zeta(t,x').$$
	And then the regularity of $\zeta$ in the dissipation $\D(t)$ is improved half order higher than one to the case of the incompressible Navier-Stokes equations with free surface in \cite{GuoY3}.
\end{remark}

\par As for the case of horizontal infinite domain, we can establish the global existence and uniqueness of strong solution. Due to the loss of the Poincar\'e inequality on the free surface $\zeta$ in the dissipation $\D(t)$, only the algebraic decay rate is showed.

\begin{theorem}(\textbf{Horizontal whole space.})\label{th2}
Assume that the initial data $(v_0,\zeta_0)\in H^4(\Omega_0)\times H^{\frac{9}{2}}(\mathbb{R}^2)$ and \eqref{fbass} and \eqref{compatibility} hold. Then there exists a small constant $\delta_0>0$, such that if the initial data satisfy
	$$\E^2(0)+\F^2(0)\leq \delta_0,$$
then the free boundary value problem \eqref{p3}-\eqref{IC} has a unique global strong solution
$$(v,w,\nn\zeta)\in L^\infty\left([0,\infty),H^4(\Omega)\times H^3(\Omega)\times H^{\frac{7}{2}}(\mathbb{R}^2)\right)\cap L^2\left([0,\infty),H^5(\Omega)\times H^4(\Omega)\times H^{\frac{7}{2}}(\mathbb{R}^2)\right)$$
with $\zeta\in L^\infty([0,\infty),H^4(\mathbb{R}^2))$,
satisfying
	$$\G(t)\leq C_2(\E^2(0)+\F^2(0)),\qquad  t\geq0$$
for some constant $C_2>0$ independent of $\delta_0$ and the time $t$.\\
Furthermore, if the initial data also satisfy $|\nabla_*|^{-\gamma}v_0\in L^2(\Omega_0)$ and $\zeta_0\in H^{-\gamma}(\R^2)$ for some fixed constant $\gamma\in (0,1)$, then the solution converges at the algebraic rate to the steady state $(\bar{v},\bar{w},\bar{\zeta},\bar{P})=(0,0,0,P_0)$, i.e.,
	$$\E(t)\leq C_3(1+t)^{-\frac{\gamma}{2}},\qquad  t\geq0,$$
where $\Omega_0$ denotes the initial domain defined in \eqref{ID} and the constant $C_3>0$ depends only on the initial data and $\gamma$.
\end{theorem}

\par Note that the definition of $|\nn|^{-\gamma}\phi$ is given in Section 3 for any function $\phi$. Let us explain the strategies to prove the above two theorems.\\

\textbf{Difficulty and Observation.}
\emph{One of difficulties encountered in the present paper is to esitmate the norm $|\nn\zeta|_{\frac{7}{2}}$ in \eqref{DED}, which is applied to control the nonlinear terms. The same estimates are also needed under consideration of the incompressible Navier-Stokes equations with free surface in \cite{GuoY3}, which are controlled via the kinematic boundary condition
	\begin{equation}\label{DDD1}
	\p_t\zeta+v\cdot\n_*\zeta=w,
	\end{equation}
and bounded by the time increase rate $C(1+t)^{\frac{1}{2}}$ for some constant $C>0$ independent of $t$. However, in the case of the primitive equations with free surface, due to the simplified vertical momentum equation, the vertical velocity $w$ is derived by the incompressible condition $\eqref{p3}$
	$$	w(t,x',x_3)=-\int_{-b}^{x_3}(J\nna\cdot v)(t,x',\tau)d\tau,$$
which causes the regularity of the vertical velocity $w$ one order lower than the horizontal velocity $v$. So the estimate $|\nn\zeta|_{\frac{7}{2}}$ can't be obtained directly from \eqref{DDD1}.}

\par \emph{Indeed, to overcome this difficulty, we make use of the relationship $\eqref{p5}$ between the pressure $P$ and the free surface $\zeta$ and the viscosity in horizontal momentum equations, and therefore integrating over the vertical direction and combining the kinematic boundary condition \eqref{DDD1} together, we get that the free surface $\zeta$ satisfies a new wave equation with the strong dissipative term
	\begin{equation}\label{DDD2}
	\p_t^2\zeta-gb\Delta_*\zeta-\Delta_*\p_t\zeta=\text{ rest terms}.
	\end{equation}
	By the equation \eqref{DDD2}, the regularity of $\zeta$ can be improved and then the estimate $|\nn\zeta|_{\frac{7}{2}}$ is bounded by the initial data, uniformly with respect to the time $t$.}\\

\par The rest parts of the paper are arranged as follows. We first establish the uniform a-priori estimates in Section 2, and then together with the local existence in the appendix, Theorem \ref{th1} is immediately proved. In Section 3, we prove Theorem \ref{th2} and get the algebraic decay rate by the interpolation inequalities. The local existence and some useful tools are listed in the appendix.\\

\noindent\textbf{Notations.} Let $A$ and $B$ be two operators, and we denote the commutator between $A$ and $B$ by $[A,B]=AB-BA$. $a\lesssim b$ means that there exists the constant $C$ independent of the time $t$ and $\delta$, such that
$a\leq C\cdot b$. The norm of Sobolev space $H^{s}(\Omega)$ or $H^r(\Gamma)$ writes $\|\cdot\|_s$ or $|\cdot|_r$, respectively and if $s=0$ or $r=0$, it means the norm of $L^2(\Omega)$ or $L^2(\Gamma)$. $\|\cdot\|_{L^{\infty}}$ or $|\cdot|_{L^{\infty}}$ denotes the norm of $L^{\infty}(\Omega)$ or $L^{\infty}(\Gamma)$, respectively. The multi index $\alpha=(\alpha_0,\alpha_1,\alpha_2,\alpha_3)\in\mathbb{N}^4$ and $|\alpha|=2\alpha_0+\alpha_1+\alpha_2+\alpha_3$. The two multi indices $\alpha\leq\beta$ means $\alpha_i\leq\beta_i$ for $i=0,1,2,3$, where $\beta=(\beta_0,\beta_1,\beta_2,\beta_3)$. For any function $\phi$, $\phi_{\alpha}$ denotes $$\phi_{\alpha}:=D^{\alpha}\phi=\p_t^{\alpha_0}\p_1^{\alpha_1}\p_2^{\alpha_2}\p_3^{\alpha_3}\phi.$$

\section{A-priori estimates}
\par In this section, the a-priori estimates on the solution $(v,\zeta)$ to the free boundary value problem \eqref{p3}-\eqref{IC}  are obtained either in the horizontal periodic domain or in horizontal infinite domain under the assumption
\begin{equation}\label{a-priori assumption}
\sup_{0\leq t\leq T}(\E^2(t)+\F^2(t))\leq \delta
\end{equation}
for some fixed constant $T>0$ and $\delta>0$ small enough which will be determined later. And then the global existence will be obtained by combining the a-priori estimates and the local existence results together.
\begin{theorem}\label{a-priori}
Let $T>0$ and $(v,w,\zeta)$ be the strong solution to the free boundary value problem \eqref{p6}-\eqref{p61} and \eqref{IC}. Suppose the assumptions in Theorem \ref{th1} or Theorem \ref{th2} hold. Then there exists a small constant $\delta>0$, such that under the a-priori assumption \eqref{a-priori assumption} for $(v,w,\zeta)$, it holds
\begin{equation}\label{a-prior-est}
\sup_{0\leq t\leq T}(\E^2(t)+\F^2(t))+\int_{0}^{T}\D^2(t)dt\leq C(\E^2(0)+\F^2(0))
\end{equation}
and for some constant $\vartheta>0$
\begin{equation}\label{ETF}
\frac{d}{dt}\E^2(t)+\vartheta\D^2(t)\leq 0,
\end{equation}
where the positive constant $C$ and $\vartheta$ are independent of the time $T$ and the constant $\delta$.
\end{theorem}

\par In the horizontal periodic case, Theorem \ref{th1} can be immediately obtained as follows, if we first admit Theorem \ref{a-priori} holds.
\begin{proof}[\textbf{Proof of Theorem \ref{th1}.}]
Due to the mean zero conditon \eqref{fbass} on the initial surface $\zeta_0$, the mean of the free surface $\zeta$ also equals zero via the incompressible condition, seeing \eqref{fbass1}. And then Poincar\'e inequality holds in the horizontal periodic domain, which implies
$$|\zeta|_0\leq C|\nn \zeta|_0.$$
Together with the definition of $\E(t)$ and $\D(t)$, it holds
$$\E(t)\leq C\D(t)$$
with the constant $C>0$ independent of the time $t$ and $\delta$.
Combining the a-priori estimate in Theorem \ref{a-priori} and the local existence Proposition \ref{local existence} in the appendix together, the global existence is obtained and the exponential decay rate is directly got by \eqref{ETF}, which completes the proof of Theorem \ref{th1}.
\end{proof}

\par In the rest of this section, the proof of Theorem \ref{a-priori} is decomposed into four parts: temporal estimates, tangential estimates, normal estimates and estimates of the free surface.

\subsection{Temporal estimates}
\par In this subsection, we prove the temporal estimates of the horizontal velocity and the free boundary via the equations \eqref{p8} in the following.

\par We first derive a lemma on any strong solution $(U,W,\eta)$, satisfying
\begin{equation}\label{p7}
\left\{\begin{aligned}
&\p_tU-\p_t\theta K\p_{3}U+v\cdot\nna U+wK\p_3U-\Delta_{\mathcal{A}} U+g\nn\eta +f\vec{\kappa}\times U=F_1,\\
&\nna\cdot U+K\p_3W=F_2
\end{aligned}\right.
\end{equation}
in $\Omega$ with the boundary conditions
\begin{equation}\label{p71}
\left\{\begin{aligned}
&\vec{n}\cdot\na U=F_3\qquad &\text{on } \Gamma,\\
&\p_t\eta+U\cdot\nn\zeta=W+F_4\qquad&\text{on }\Gamma,\\
&U=W=0 \qquad&\text{on }\Sigma_b,
\end{aligned}\right.
\end{equation}
where $U:=(U^1,U^2)$ and $v,w$ and $\zeta$ are the solution to \eqref{p6}-\eqref{p61}.
\begin{lemma}\label{prop_t}
	Let $T>0$ and $ (U,W,\eta) $ be the regular solution to equations \eqref{p7}-\eqref{p71} for $(x,t)\in\Omega\times(0,T]$. Then, it holds
	\begin{equation*}
	\begin{aligned}
	&\frac{1}{2}\frac{d}{dt}\left[\int_{\Omega}J|U|^2dx+
	\int_{\Gamma}g|\eta|^2dx'\right]+\int_{\Omega}J|\na U|^2dx
	=\int_{\Omega}J[F_1\cdot U+ F_2\cdot g\eta]dx
	+\int_{\Gamma}[F_3\cdot U+F_4\cdot g\eta]dx'.
	\end{aligned}
	\end{equation*}	
\end{lemma}
\begin{proof}
The equality is directly obtained via multiplying $\eqref{p7}_1$ by $JU$, integrating over the domain $\Omega$ by parts, and then combining $\eqref{p7}_2$ and the boundary condition \eqref{p71} together. The details are omitted.
\end{proof}

\par By Lemma \ref{prop_t}, we have the following temporal estimates.
\begin{proposition}\label{TE}
Let $T>0$ and $(v,w,\zeta)$ be the strong solution to equations \eqref{p6}-\eqref{p61}. Suppose the assumptions in Theorem \ref{th1} or Theorem \ref{th2} hold. Then, under the a-priori assumption \eqref{a-priori assumption}, we have
\begin{equation}\label{p19}
\frac{1}{2}\frac{d}{dt}\sum_{i=0}^{2}\left[\int_{\Omega}J|\p_t^iv|^2dx+
\int_{\Gamma}g|\p_t^i\zeta|^2dx'\right]+\sum_{i=0}^{2}\int_{\Omega}J|\na \p_t^iv|^2dx\lesssim \E(t)\D(t)^2,
\end{equation}
where the energy $\E(t)$ and dissipation $\D(t)$ are defined by \eqref{DED}.
\end{proposition}
\begin{proof}
To obtain the temporal estimates for the strong solution to equations \eqref{p6}-\eqref{p61} for $(x,t)\in\Omega\times(0,T]$, we differentiate the equations \eqref{p6}-\eqref{p61} directly with respect to time $t$ to have the following equations
\begin{equation}\label{p8}
\left\{\begin{aligned}
&\p_tv_{\alpha_0}-\p_t\theta K\p_{3}v_{\alpha_0}+v\cdot\nna v_{\alpha_0}+wK\p_3v_{\alpha_0}-\Delta_{\mathcal{A}} v_{\alpha_0}+g\nn\zeta_{\alpha_0} +f\vec{\kappa}\times v_{\alpha_0}=F_1^{\alpha_0},\\
&\nna\cdot v_{\alpha_0}+K\p_3w_{\alpha_0}=F_2^{\alpha_0}
\end{aligned}\right.
\end{equation}
for $(x,t)\in\Omega\times(0,T]$ with the boundary conditions
\begin{equation}\label{p81}
\left\{\begin{aligned}
&\vec{n}\cdot\na v_{\alpha_0}=F_3^{\alpha_0}\qquad &\text{on } \Gamma,\\
&\p_t\zeta_{\alpha_0}+v_{\alpha_0}\cdot\nn\zeta=w_{\alpha_0}+F_4^{\alpha_0}\qquad&\text{on }\Gamma,\\
&v_{\alpha_0}=w_{\alpha_0}=0 \qquad&\text{on }\Sigma_b,
\end{aligned}\right.
\end{equation}
where $V_{\alpha_0}:=\p_t^{\alpha_0}V$ for any regular function $V$ with $\alpha_0=0,1,2$ and the nonliear terms $F_i^{\alpha_0}\ (i=1,\cdots,4)$ are defined by
\begin{equation} \label{N-al}
\begin{aligned}
F_1^{\alpha_0}=&\p_t^{\alpha_0}(\p_t\theta K\p_3v)-\p_t\theta K\p_3v_{\alpha_0}-\p_t^{\alpha_0}(v\cdot\nna v+wK\p_3 w)\\
&\ \ +v\cdot\nna v_{\alpha_0}+wK\p_3w_{\alpha_0}+\p_t^{\alpha_0}(\Delta_{\mathcal{A}}v)-\Delta_{\mathcal{A}}v_{\alpha_0},\\
F_2^{\alpha_0}=&-\p_t^{\alpha_0}(\nna\cdot v+K\p_3w)+\nna\cdot v_{\alpha_0}+K\p_3\p_t^{\alpha_0}w,\\
F_3^{\alpha_0}=&-\p_t^{\alpha_0}(\vec{n}\cdot\na v)+\vec{n}\cdot\na v_{\alpha_0},\\
F_4^{\alpha_0}=&-\p_t^{\alpha_0}(v\cdot\nn\zeta)+v_{\alpha_0}\cdot\nn\zeta.
\end{aligned}\end{equation}

\par Multiplying $\eqref{p8}_1$ with $Jv_{\alpha_0}$ and integrating the resulted equation by part over $\Omega$, we have by Lemma \ref{prop_t} that
\begin{equation}\label{p9}
\begin{aligned}
&\frac{1}{2}\frac{d}{dt}\left[\int_{\Omega}J|v_{\alpha_0}|^2dx+
\int_{\Gamma}g|\zeta_{\alpha_0}|^2dx'\right]+\int_{\Omega}J|\na v_{\alpha_0}|^2dx\\
=&\int_{\Omega}J[F_1^{\alpha_0}\cdot v_{\alpha_0}+ F_2^{\alpha_0}\cdot g\zeta_{\alpha_0}]dx
+\int_{\Gamma}[F_3^{\alpha_0}\cdot v_{\alpha_0}+F_4^{\alpha_0}\cdot g\zeta_{\alpha_0}]dx'\\
:=&I^{\alpha_0}_1+I^{\alpha_0}_2+I^{\alpha_0}_3+I^{\alpha_0}_4.
\end{aligned}
\end{equation}	
The nonlinear terms in the right hand side of \eqref{p9} can be estimated below for $\alpha_0=0,1,2$ respectively.

Indeed, if $\alpha_0=0$, the definition of $F^{\alpha_0}_i$ gives
$$F^{\alpha_0}_i=0 \qquad \text{for }1\leq i\leq4.$$
Then we have
\begin{equation}\label{p171}
\frac{1}{2}\frac{d}{dt}\left[\int_{\Omega}J|v|^2dx+
\int_{\Gamma}g|\zeta|^2dx'\right]+\int_{\Omega}J|\na v|^2dx=0,
\end{equation}
and since the following estimates on the nonlinear terms $I_j^1$, $j=1,2,3,4$ of \eqref{p9} hold for $\alpha_0=1$,
\begin{equation*}
\begin{aligned}
|I^1_1|\lesssim& (1+|\nn\zeta|_2)^2\left[(1+|\p_t\zeta|_2)|\p_t\zeta|_2\|v\|_4\|\p_tv\|_0+|\p_t^2\zeta|_0\|v\|_4\|\p_tv\|_0+\|v\|_4\|\p_tv\|^2_2\right]\\
\lesssim &\E(t)\D(t)^2,\\
|I^1_2|\lesssim&(1+|\nn\zeta|_2)|\p_t\zeta|_2^2\|v\|_4\lesssim \E(t)\D(t)^2,\\
|I^1_3|\lesssim&(1+|\nn\zeta|_2)|\p_t\zeta|_1\|\p_tv\|_2\|v\|_4\lesssim \E(t)\D(t)^2,\\
|I^1_4|\lesssim&\|v\|_4|\p_t\zeta|^2_0\lesssim \E(t)\D(t)^2,
\end{aligned}
\end{equation*}
we obtain
\begin{equation}\label{p17}
\frac{1}{2}\frac{d}{dt}\left[\int_{\Omega}J|v_{\alpha_0}|^2dx+
\int_{\Gamma}g|\zeta_{\alpha_0}|^2dx'\right]+\int_{\Omega}J|\na v_{\alpha_0}|^2dx\lesssim \E(t)\D(t)^2\qquad \text{for }\alpha_0=1.
\end{equation}

\par What left is to derive the expected estimates on $I_j^2$ ($j=1,2,3,4$) of \eqref{p9} for $\alpha_0=2$ as follows. The estimates of $I_1^2$ consists of three parts. By the definition of $F^2_1$, it holds
$$F_1^{2}=\sum_{i=0}^{1}\left(\begin{matrix}2\\i\end{matrix}\right)\left[\p_t^{2-i}(\p_t\theta K)\p_t^i\p_3v
-\p_t^{2-i}(v\cdot\nna)\p_t^iv-\p_t^{2-i}(wK)\p_t^i\p_3v
+\p_t^{2-i}(\Delta_{\mathcal{A}})\p_t^iv
\right],$$
and therefore we have for $0\leq i\leq 1$ that
\begin{equation}\label{p10}
\begin{aligned}
&\left|\int_{\Omega}J\p_t^{2-i}(\p_t\theta K)\p_t^i\p_3v\cdot\p_t^2vdx\right|\\
\lesssim& \|J\|_{L^{\infty}}\cdot \left(\sum_{j=0}^{1-i}\|\p_t^{3-i-j}\theta\|_0\cdot\|\p_t^jK\|_{L^{\infty}}+\|\p_t\theta\|_{L^{\infty}}\cdot\|\p_t^{2-i}K\|_{0}\right)\cdot\|\p_t^i\p_3v\|_{L^{\infty}}\cdot\|\p_t^2v\|_{0}\\
\lesssim &\E(t)\D(t)^2
\end{aligned}
\end{equation}
and
\begin{equation}\label{p11}
\begin{aligned}
&\left|\int_{\Omega}J\left(\p_t^{2-i}(v\cdot\nna)\p_t^iv+\p_t^{2-i}(wK)\p_t^i\p_3v\right)\cdot \p_t^2vdx\right|\\
\lesssim &\|J\|_{L^{\infty}} \left(\sum_{j=0}^{1-i}(\|\p_t^{2-i-j}v\|_0
+\|\p_t^{2-i-j}w\|_0)(1+\|\n\p_t^j\theta\|_{L^{\infty}})+(\|v\|_{L^{\infty}}+\|w\|_{L^{\infty}})\|\n\p_t^{2-i}\theta\|_0\right)\\
&\quad\times\|\p_t^i\n v\|_{L^{\infty}}\cdot\|\p_t^2v\|_0\\
\lesssim&\E(t)\D(t)^2,
\end{aligned}\end{equation}
where we have made use of the relation \eqref{p5} to control the vertical velocity $w$ in terms of the horizontal velocity $v$.  %\\
Since the last term in $F^2_1$ can be rewritten as
$$\begin{aligned}
\p_t^{2-i}(\Delta_{\mathcal{A}})\p_t^iv
=&\sum_{l,k,m=0}^{3}\left[\p_t^{2-i}\left(\a_{lk}\p_k\a_{lm}\right)\p_t^i\p_mv+\p_t^{2-i}\left(\a_{lk}\a_{lm}\right)\p_t^i\p_k\p_mv\right]\\
=&\sum_{j=0}^{2-i}\sum_{l,k,m=0}^{3}\left(\begin{matrix}2-i\\j\end{matrix}\right)\left[\p_t^{2-i-j}\a_{lk}\p_t^j\p_k\a_{lm}\p_t^i\p_mv+\p_t^{2-i-j}\a_{lk}\p_t^j\a_{lm}\p_t^i\p_k\p_mv\right],
\end{aligned}$$
we have for $i=0$ that
\begin{equation}\label{D1}
\begin{aligned}
&\left|\int_{\Omega}J\p_t^2(\Delta_{\mathcal{A}})v\cdot\p_t^2vdx\right|\\
\lesssim& \|J\|_{L^{\infty}}\cdot\Bigg[\left(\|\p_t^2\n\theta\|_0\cdot\|\n^2\theta\|_{L^{\infty}}+\sum_{j=1}^{2}(\|\p_t^{2-j}\n\theta\|_{L^{\infty}}\cdot\|\p_t^{j}\n^2\theta\|_0)\right)\cdot\|\n v\|_{L^{\infty}}\\
&+\left(\|\p_t^2\n\theta\|_0\cdot(1+\|\n\theta\|_{L^{\infty}})+\sum_{j=1}^{2}(\|\p_t^{2-j}\n\theta\|_{L^{\infty}}\cdot\|\p_t^{j}\n\theta\|_0)\right)\cdot\|\n^2 v\|_{L^{\infty}}\Bigg]\cdot \|\p_t^2v\|_0\\
\lesssim&\E(t)\D(t)^2,
\end{aligned}
\end{equation}
and for $i=1$ that
\begin{equation}\label{D2}
\begin{aligned}
\left|\int_{\Omega}J\p_t(\Delta_{\mathcal{A}})\p_tv\cdot\p_t^2vdx\right|
\lesssim& \|J\|_{L^{\infty}}\cdot\Big[\sum_{j=0}^{1}\|\p_t^{1-j}\n\theta\|_{L^{\infty}}\cdot\|\n^2\p_t^j\theta\|_{L^{\infty}}\cdot\|\n\p_t v\|_0\\
&+\|\p_t\n\theta\|_{L^{\infty}}\cdot(1+\|\n\theta\|_{L^{\infty}})\cdot\|\n^2\p_t v\|_0\Big]\cdot \|\p_t^2v\|_0\\
\lesssim&\E(t)\D(t)^2.
\end{aligned}
\end{equation}
The summation of \eqref{D1} and \eqref{D2} leads to
\begin{equation}\label{p12}
\left|\int_{\Omega}\p_t^{2-i}(\Delta_{\mathcal{A}})\p_t^iv\cdot\p_t^2vdx\right|\lesssim \E(t)\D(t)^2\qquad\text{for }i=0,1.
\end{equation}
Combining \eqref{p10}, \eqref{p11} and \eqref{p12} together, we obtain
\begin{equation}\label{p13}
\left|I_1^2\right|\lesssim \E(t)\D(t)^2.
\end{equation}

\par By the definition of $F_2^2$ below
$$F_2^2=-\sum_{i=0}^{1}\left(\begin{matrix}2\\i\end{matrix}\right)\left[(\p_t^{2-i}\a\n)_*\cdot\p_t^iv+\p_t^{2-i}K\p_3\p_t^iw\right],$$
the term $I_2^2$ can be estimated by
\begin{equation}\label{p14}
\begin{aligned}
\left|I_2^2\right|
\lesssim& \|J\|_{L^{\infty}}\cdot \left(\|\n\p_t^2\theta\|_0\cdot (\|\n v\|_{L^{\infty}}+\|\n w\|_{L^{\infty}})+ \|\n\p_t\theta\|_{L^{\infty}}\cdot(\|\n \p_tv\|_0+\|\n\p_tw\|_0)\right)\cdot\|\p_t^2\zeta\|_0\\
\lesssim&\E(t)\D(t)^2.
\end{aligned}\end{equation}

\par By the definition of $F^2_3$ below
$$\begin{aligned}
F_3^{2}&=-\p_t^{2}(\vec{n}\cdot\na)v-2\p_t(\vec{n}\cdot\na)\p_tv\\
&=-\sum_{j=0}^{2}\left(\begin{matrix}2\\j\end{matrix}\right)\p_t^j\vec{n}\cdot(\p_t^{2-j}\a\n)v-2\left(\p_t\vec{n}\cdot\na+\vec{n}\cdot(\p_t\a\n)\right)\p_tv,
\end{aligned}$$
where $\vec{n}=(-\p_1\zeta,-\p_2\zeta,1)^T$, the term $I_3^2$ can be estimated by
\begin{equation}\label{p15}
\begin{aligned}
\left|I_3^2\right|
\lesssim&|J|_{L^{\infty}}\cdot\Bigg[\left(\sum_{j=0}^{1}(1+|\p_t^j\nn\zeta|_{L^{\infty}})\cdot|\p_t^{2-j}\n\theta|_0+|\p_t^2\nn\zeta|_0\cdot(1+|\n\theta|_{L^{\infty}})\right)\cdot|\n v|_{L^{\infty}}\\
&+\left(|\p_t\nn\zeta|_{L^{\infty}}\cdot(1+|\n\theta|_{L^{\infty}})
+(1+|\nn\zeta|_{L^{\infty}})\cdot|\p_t\n\theta|_{L^{\infty}}\right)\cdot|\n\p_tv|_0\Bigg]\cdot |\p_t^2v|_0\\
\lesssim &\E(t)\D(t)^2,
\end{aligned}\end{equation}
where the Sobolev embedding inequalities and Lemma \ref{SME} are used to bound the nonlinear terms.

\par Similarly, by the definition of $F^2_4$, we can control the term $I_4^2$ as
\begin{equation}\label{p16}
\begin{aligned}
\left|I_4^2\right|=&\left|\int_{\Gamma}J\cdot[-2\p_tv\cdot\nn\p_t\zeta-v\cdot\nn\p_t^2\zeta]\cdot g\p_t^2\zeta dx'\right|\\
\lesssim&
|J|_{L^{\infty}}\cdot\left[\sum_{i=0}^{1}|\p_t^iv|_{L^{\infty}}|\nn\p_t^{2-i}\zeta|_0\right]\cdot |\p_t^2\zeta|_0\\
\lesssim&\E(t)\D(t)^2.
\end{aligned}\end{equation}

\par In summary, combining \eqref{p9} with \eqref{p13}-\eqref{p16} together, we get for $\alpha_0=2$ that
\begin{equation}\label{p18}
\frac{1}{2}\frac{d}{dt}\left[\int_{\Omega}J|\p_t^2v|^2dx+
\int_{\Gamma}g|\p_t^2\zeta|^2dx'\right]+\int_{\Omega}J|\na \p_t^2v|^2dx\lesssim \E(t)\D(t)^2.
\end{equation}
By adding \eqref{p171}, \eqref{p17} and \eqref{p18} together, we can establish the temporal estimate \eqref{p19}.
\end{proof}

\subsection{Tangential estimates}
\par In this subsection, the tangential estimates on the horizontal velocity and the free boundary are established by the linearized system of equations \eqref{p6} and \eqref{p61}
\begin{equation}\label{pl1}
\left\{\begin{aligned}
&\p_tv+g\nn\zeta-\Delta v+f\vec{\kappa}\times v=G_1,\\
&\nn\cdot v+\p_3w=G_2
\end{aligned}\right.
\end{equation}
in $\Omega$ with the boundary condition
\begin{equation}\label{pl2}
\left\{\begin{aligned}
&\p_3v=G_3 \qquad& \text{on } \Gamma,\\
&\p_t\zeta=w+G_4 \qquad& \text{on } \Gamma,\\
&v=w=0\qquad& \text{on } \Sigma_b,
\end{aligned}\right.
\end{equation}
where the nonlinear terms $G_i$ are defined by
\begin{equation}\label{NDDD}
\begin{aligned}
&G_1:=\p_t\theta K\p_3v-v\cdot\nna v-wK\p_3v+\Delta_{\mathcal{A}}v-\Delta v,\\
&G_2:=-\nna\cdot v-K\p_3w+\nn\cdot v+\p_3w,\\
&G_3:=-\vec{n}\cdot\na v+\p_3v,\\
&G_4:=-v\cdot\nn \zeta.
\end{aligned}
\end{equation}
\par We have the tangential estimates about the solution to equations \eqref{pl1}-\eqref{NDDD}.
\begin{proposition}\label{Tang E}
	Let $T>0$ and $(v,w,\zeta)$ be the strong solution to equations \eqref{p6}-\eqref{p61}. Suppose the assumptions in Theorem \ref{th1} or Theorem \ref{th2} hold. Then, under the a-priori assumption \eqref{a-priori assumption}, we have
	\begin{equation}\label{pt}
	\frac{1}{2}\frac{d}{dt}\sum_{\substack{|\alpha|\leq 4\\0\leq\alpha_0<2}}\left[\int_{\Omega}|v_{\alpha}|^2dx+\int_{\Gamma}g|\zeta_{\alpha}|^2dx'\right]+\sum_{\substack{|\alpha|\leq 4\\0\leq\alpha_0<2}}\int_{\Omega}|\n v_{\alpha}|^2dx\lesssim \left(\E(t)+\F(t)\right)\D(t)^2,
	\end{equation}
	where $\E(t)$, $\D(t)$ and $\F(t)$ are defined by \eqref{DED} and $V_{\alpha}:=D^{\alpha}V=\p_t^{\alpha_0}\p_1^{\alpha_1}\p_2^{\alpha_2}V$ for any regular function $V$ with the index $\alpha:=(\alpha_0,\alpha_1,\alpha_2,0)$ satisfying $|\alpha|=2\alpha_0+\alpha_1+\alpha_2\leq 4$ and $0\leq\alpha_0<2$.
\end{proposition}
\begin{proof}
	To obtain the tangential estimates for the strong solution to equations \eqref{p6}-\eqref{p61} for $(x,t)\in\Omega\times (0,T]$, we differentiate the equations \eqref{pl1}-\eqref{pl2} directly with respect to the time $t$ and the horizontal direction $x'=(x_1,x_2)$ to have the following equations
	\begin{equation}\label{pl3}
	\left\{\begin{aligned}
	&\p_tv_{\alpha}+g\nn\zeta_{\alpha}-\Delta v_{\alpha}+f\vec{\kappa}\times v_{\alpha}=D^{\alpha}G_1\qquad &\text{in }\Omega,\\
	&\nn\cdot v_{\alpha}+\p_3w_{\alpha}=D^{\alpha}G_2\qquad &\text{in }\Omega
	\end{aligned}\right.
	\end{equation}
	and
	\begin{equation}\label{pl4}
	\left\{\begin{aligned}
	&\p_3v_{\alpha}=D^{\alpha}G_3 \qquad& \text{on } \Gamma,\\
	&v_{\alpha}=w_{\alpha}=0\qquad& \text{on } \Sigma_b,\\
	&\p_t\zeta_{\alpha}=w_{\alpha}+D^{\alpha}G_4 \qquad& \text{on } \Gamma,
	\end{aligned}\right.
	\end{equation}
	where $\alpha:=(\alpha_0,\alpha_1,\alpha_2,0)$ satisfies $|\alpha|\leq 4$ and $0\leq\alpha_0<2$.
	
	\par Multiplying $\eqref{pl3}_1$ with $v_{\alpha}$ and integrating the resulted equation by parts over the domain $\Omega$, we get
	\begin{equation}\label{pl5}
	\begin{aligned}
	&\frac{1}{2}\frac{d}{dt}\left[\int_{\Omega}|v_{\alpha}|^2dx+\int_{\Gamma}g|\zeta_{\alpha}|^2dx'\right]+\int_{\Omega}|\n v_{\alpha}|^2dx\\
	=&\int_{\Omega}\left[D^{\alpha}G_1\cdot v_{\alpha}+D^{\alpha}G_2\cdot g\zeta_{\alpha}\right]dx+\int_{\Gamma}\left[D^{\alpha}G_3\cdot v_{\alpha}+D^{\alpha}G_4\cdot g\zeta_{\alpha}\right]dx'\\
	:=&II^{\alpha}_1+II^{\alpha}_2+II^{\alpha}_3+II^{\alpha}_4.
	\end{aligned}\end{equation}
	The nonlinear terms in the right hand side of \eqref{pl5} can be estimated below for $\alpha_0=0$ and $\alpha_0=1$ respectively.
	
\par If $\alpha_0=0$ and $1\leq\alpha_1+\alpha_2\leq 4$, the term $II^{\alpha}_1$ is bounded by the definition of $G_1$ in \eqref{NDDD}
$$II^{\alpha}_1
=\int_{\Omega}D^{\alpha}\left[\p_t\theta K\p_3v-v\cdot\nna v-wK\p_3v\right]\cdot v_{\alpha}dx+\int_{\Omega}D^{\alpha}\left[\Delta_{\mathcal{A}}v-\Delta v\right]\cdot v_{\alpha}dx:=II_5+II_6,$$
and integrating by parts, we get for any $1\leq\alpha_1+\alpha_2\leq 4$
\begin{equation}\label{pl6}
\begin{aligned}
|II_5|&=\left|-\int_{\Omega}D^{\alpha+\beta}v\cdot D^{\alpha-\beta}\left[\p_t\theta K\p_3v-v\cdot\nna v-wK\p_3v\right]dx\right|\\
&\lesssim \|D^{\alpha+\beta}v\|_0\cdot\|\p_t\theta K\p_3v-v\cdot\nna v-wK\p_3v\|_3\\
&\lesssim \E(t)\D(t)^2,
\end{aligned}
\end{equation}
where the index $\beta=(0,\beta_1,\beta_2,0)$ satisfies
\begin{equation}\label{K1}
\beta\leq \alpha\text{ and }|\beta|=1.
\end{equation}
Note that by the relation \eqref{p5}, the vertical velocity $w$ can be bounded as
$$\|w\|_3=\left\|\int_{-b}^{x_3}(J\nna\cdot v)(t,x',s)ds\right\|_3
\lesssim\|J\nna\cdot v\|_3\lesssim (1+|\zeta|_4)^2\|v\|_4\lesssim\|v\|_4.$$

To derivate the expected estimates on $II_6$, by direct computation, we have
\begin{equation}\label{t1}
\begin{aligned}
\Delta_{\mathcal{A}}v-\Delta v=&-\p_1(AK)\p_3v-2AK\p_1\p_3v+AK\p_3(AK)\p_3v+A^2K^2\p_3^2v-\p_2(BK)\p_3v\\
&-2BK\p_2\p_3v+BK\p_3(BK)\p_3v+B^2K^2\p_3^2v-K^3\p_3^2\theta\p_3v-K^2(2\p_3\theta+|\p_3\theta|^2)\p_3^2v,
\end{aligned}\end{equation}
and therefore $II_6$ is estimated by
\begin{equation}\label{pl7}\begin{aligned}
|II_6|=&\left|-\int_{\Omega}D^{\alpha+\beta}v\cdot D^{\alpha-\beta}\left(\Delta_{\mathcal{A}}v-\Delta v\right)dx\right|
\lesssim \|D^{\alpha+\beta}v\|_0\cdot\|\Delta_{\mathcal{A}}v-\Delta v\|_3\\
\lesssim&\|v\|_5\cdot\left(\|\n\theta\|_4\cdot\|v\|_4+\|\theta\|_4\cdot\|v\|_5\right)\cdot(1+\|\theta\|_4)^3\\
\lesssim& \left(\E(t)+\F(t)\right)\D(t)^2,
\end{aligned}\end{equation}
where the multi index $\beta$ is defined in \eqref{K1}.
Adding \eqref{pl6} and \eqref{pl7} together, we have
\begin{equation}\label{pe1}
\left|II^{\alpha}_1\right|=|II_5+II_6|\lesssim \left(\E(t)+\F(t)\right)\D(t)^2\quad \text{for } \alpha_0=0 \text{ and }1\leq\alpha_1+\alpha_2\leq 4.
\end{equation}

\par By the definition of $G_2$ in \eqref{NDDD} and the relation \eqref{p5}, it holds
$$G_2=AK\p_3v^1+BK\p_3v^2-\p_3\theta\nna\cdot v,$$
and then we get that for $1\leq\alpha_1+\alpha_2\leq4$
\begin{equation}\label{pe2}
\begin{aligned}
\left|II^{\alpha}_2\right|&\lesssim \|G_2\|_4\cdot|\nn\zeta|_3\\
&\lesssim
\left(\|\n\theta\|_4\cdot\|v\|_4+\|\theta\|_4\cdot\|v\|_5\right)(1+\|\theta\|_4)\cdot|\nn\zeta|_3\\
&\lesssim \left(\E(t)+\F(t)\right)\D(t)^2.
\end{aligned}\end{equation}

\par To derive the estimate on the term $II^{\alpha}_3$, the definition of $G_3$ in \eqref{NDDD} gives
\begin{equation}\label{C1}
G_3=\nn\zeta\cdot\nn v+K(-|\nn\zeta|^2+\p_3\theta)\p_{3}v,
\end{equation}
and then integrating by parts, by Lemma \ref{SME} in the appendix, we obtain that for $1\leq\alpha_1+\alpha_2\leq4$
\begin{equation}\label{pe3}
\begin{aligned}
\left|II^{\alpha}_3\right|&=\left|-\int_{\Gamma}D^{\alpha-\beta}G_3\cdot D^{\alpha+\beta}vdx'\right|\lesssim |D^{\alpha-\beta}G_3|_{\frac{1}{2}}\cdot|v_{\alpha+\beta}|_{-\frac{1}{2}}\\
&\lesssim \left(|\nn\zeta|_{\frac{7}{2}}\cdot\|v\|_5+|\zeta|_{4}\cdot\|v\|_4\right)\cdot(1+|\zeta|_4)^2\cdot \|v\|_5\\
&\lesssim\left(\E(t)+\F(t)\right)\D(t)^2,
\end{aligned}\end{equation}
where the multi-index $\beta$ is defined in \eqref{K1}.

\par To estimate the term $II^{\alpha}_4$, by the definition of $G_4$ in \eqref{NDDD}, we have
$$D^{\alpha}G_4=-\sum_{\gamma<\alpha}C_{\alpha,\gamma}D^{\alpha-\gamma}v\cdot\nn D^{\gamma}\zeta-v\cdot\nn D^{\alpha}\zeta:=II_7+II_8,$$
where $C_{\alpha,\gamma}$ is the constant only depending on the multi indices $\alpha$ and $\gamma$. Then we get that for $1\leq\alpha_1+\alpha_2\leq4$
\begin{equation*}
\left|\int_{\Gamma}II_7\cdot g\zeta_{\alpha}dx' \right|\lesssim \|v\|_5\cdot|\nn\zeta|_3\cdot|\zeta_{\alpha}|_0\lesssim \E(t)\D(t)^2
\end{equation*}
and via integrating by parts
$$\left|\int_{\Gamma}II_8\cdot D^{\alpha}\zeta dx'\right|=\left|-\frac{1}{2}\int_{\Gamma}\nn\cdot v |D^{\alpha}\zeta|^2dx'\right|\lesssim \|v\|_4\cdot|\nn\zeta|_3^2\lesssim \E(t)\D(t)^2,$$
so adding these two inequalities together, we get
\begin{equation}\label{pe4}
\left|II^{\alpha}_4\right|\lesssim \E(t)\D(t)^2.
\end{equation}

\par Combining \eqref{pe1}, \eqref{pe2}, \eqref{pe3} and \eqref{pe4} together, we obtain
\begin{equation}\label{pt0}
\frac{1}{2}\frac{d}{dt}\left[\int_{\Omega}|v_{\alpha}|^2dx+\int_{\Gamma}g|\zeta_{\alpha}|^2dx'\right]+\int_{\Omega}|\n v_{\alpha}|^2dx\lesssim \left(\E(t)+\F(t)\right)\D(t)^2
\end{equation}
for $\alpha=(0,\alpha_1,\alpha_2,0)$ and $1\leq\alpha_1+\alpha_2\leq 4$.

\par If $\alpha_0=1$ and $1\leq\alpha_1+\alpha_2\leq 2$, by the definition of $G_1$ in \eqref{NDDD}, we have
$$II^{\alpha}_1=\int_{\Omega}D^{\alpha}\left[\p_t\theta K\p_3v-v\cdot\nna v-wK\p_3v\right]\cdot v_{\alpha}dx+\int_{\Omega}D^{\alpha}\left[\Delta_{\mathcal{A}}v-\Delta v\right]\cdot v_{\alpha}dx:=II_9+II_{10}.$$
Then integrating by parts, we get
\begin{equation}\label{t11}
\begin{aligned}
|II_9|&\lesssim \|D^{\alpha+\beta}v\|_0\cdot\|\p_t[\p_t\theta K\p_3v-v\cdot\nna v-wK\p_3v]\|_1\\
&\lesssim \E(t)\D(t)^2
\end{aligned}
\end{equation}
and
\begin{equation}\label{t12}
\begin{aligned}
|II_{10}|\lesssim& \|D^{\alpha+\beta}v\|_0\cdot\|\p_t(\Delta_{\mathcal{A}}v-\Delta v)\|_1\\
\lesssim&\|\p_tv\|_3\cdot\left(\|\p_t\theta\|_4\cdot\|v\|_4+\|\theta\|_4\cdot\|\p_tv\|_3\right)\cdot(1+\|\theta\|_4)^3\\
\lesssim& \E(t)\D(t)^2,
\end{aligned}
\end{equation}
where the multi index $\beta$ is defined in \eqref{K1}.
So adding \eqref{t11} and \eqref{t12} together, we obtain
\begin{equation}\label{t111}
\left|II^{\alpha}_1\right|=|II_9+II_{10}|\lesssim\E(t)\D(t)^2.
\end{equation}

\par By the direct computation and the definitions of the nonlinear terms $G_i$ \eqref{NDDD}, the term $II^{\alpha}_2$, $II^{\alpha}_3$ and $II^{\alpha}_4$ are bounded respectively as
\begin{equation}\label{t112}
\begin{aligned}
\left|II^{\alpha}_2\right|&\lesssim \|\p_tG_2\|_2\cdot|\p_t\zeta|_2\\
&\lesssim
\left(\|\p_t\theta\|_3\cdot\|v\|_4+\|\theta\|_4\cdot\|\p_tv\|_3\right)(1+\|\theta\|_4)\cdot|\p_t\zeta|_2\\
&\lesssim \E(t)\D(t)^2,
\end{aligned}\end{equation}
\begin{equation}\label{t113}
\begin{aligned}
\left|II^{\alpha}_3\right|&=\left|-\int_{\Gamma}D^{\alpha-\beta}G_3\cdot D^{\alpha+\beta}vdx'\right|\lesssim |D^{\alpha-\beta}G_3|_{\frac{1}{2}}\cdot|v_{\alpha+\beta}|_{-\frac{1}{2}}\\
&\lesssim \left(|\p_t\zeta|_{\frac{5}{2}}\cdot\|v\|_4+|\nn\zeta|_{3}\cdot\|\p_tv\|_3\right)\cdot(1+|\zeta|_4)^2\cdot \|\p_tv\|_3\\
&\lesssim\E(t)\D(t)^2,
\end{aligned}\end{equation}
and
\begin{equation}\label{t114}
\begin{aligned}
\left|II^{\alpha}_4\right|&\lesssim |\p_tG_4|_2\cdot|\p_t\zeta|_2\\
&\lesssim\left(\|\p_tv\|_3\cdot|\nn\zeta|_2+\|v\|_3\cdot|\p_t\zeta|_3\right)\cdot|\p_t\zeta|_2\\
&\lesssim\E(t)\D(t)^2,
\end{aligned}
\end{equation}
where the index $\beta$ is defined in \eqref{K1}.
So adding the estimates \eqref{t111}, \eqref{t112}, \eqref{t113} and \eqref{t114} together, we obtain that
\begin{equation}\label{pt1}
\frac{1}{2}\frac{d}{dt}\left[\int_{\Omega}|v_{\alpha}|^2dx+\int_{\Gamma}g|\zeta_{\alpha}|^2dx'\right]+\int_{\Omega}|\n v_{\alpha}|^2\lesssim \E(t)\D(t)^2
\end{equation}
for $\alpha=(1,\alpha_1,\alpha_2,0)$ and $1\leq\alpha_1+\alpha_2\leq 2$.

\par In summary, the tangential estimates \eqref{pt} is obtaind by adding \eqref{pt0} and \eqref{pt1} together.
\end{proof}

\subsection{Normal estimates}
\par In this subsection, we establish the normal estimates of the solution to the linearized equations \eqref{pl1}-\eqref{pl2}.
We firstly give an useful lemma to bound the nonlinear term $G_1$ in equations \eqref{pl1}, which is achieved by a direct computation using the Sobolev embedding inequalities and Lemma \ref{SME}. The proof is omitted here.
\begin{lemma}\label{pln1}
	For any index $\alpha=(\alpha_0,\alpha_1,\alpha_2,\alpha_3)$ satisfying $|\alpha|=2\alpha_0+\alpha_1+\alpha_2+\alpha_3\leq 2$, we have
	\begin{equation*}
	\begin{aligned}
	&\|D^{\alpha}G_1\|_0\lesssim \E(t)^2,\\
	&\|D^{\alpha}G_1\|_1\lesssim \left[\E(t)+\F(t)\right]\D(t),
	\end{aligned}
	\end{equation*}
	where $\E(t)$, $\D(t)$ and $\F(t)$ are defined by \eqref{DED}.
\end{lemma}
\par By Lemma \ref{pln1}, we have the following normal estimates.
\begin{proposition}\label{NE}
	Let $T>0$ and $(v,w,\zeta)$ be the strong solution to equations \eqref{p6}-\eqref{p61}. Suppose the assumptions in Theorem \ref{th1} or Theorem \ref{th2} hold. Then, under the a-priori assumption \eqref{a-priori assumption}, we have
	\begin{equation}\label{pnnn1}
	\sum_{i=0}^{1}\|\p_t^iv\|_{4-2i}\lesssim \E(t)^{3/2}+\sum_{|\alpha|\leq 4}\left(\|D^{\alpha}v\|_0+|D^{\alpha}\zeta|_0\right)
	\end{equation}
	and
	\begin{equation}\label{pnnn2}
	\sum_{i=0}^{2}\|\p_t^iv\|_{5-2i}+|\nn\zeta|_{3}+\sum_{i=1}^{3}|\p_t^i\zeta|_{\frac{11}{2}-2i}\lesssim \left(\E(t)^{1/2}+\F(t)^{1/2}\right)\D(t)+\sum_{|\alpha|\leq 4}\|D^{\alpha}v\|_1,
	\end{equation}
	where $\alpha=(\alpha_0,\alpha_1,\alpha_2,0)$.
\end{proposition}
\begin{proof}
	To obtain the normal estimates for the strong solution to equations \eqref{p6}-\eqref{p61}, we first derive the bound of $\zeta$ in \eqref{pnnn2}.
	Denote $\phi(x_3)\in C_0^{\infty}((-b,0))$ to a cut-off function, satisfying $$0\leq|\phi(x_3)|+|\phi'(x_3)|\leq M,\qquad \forall x_3\in (-b,0)$$
	for some fixed constant $M>0$ only dependent of $b$ and
	$$\int_{-b}^{0}\phi(x_3)dx_3=1.$$
	Then multiplying the equation \eqref{pl3} by $\phi\nn\zeta_{\alpha}$ and integrating the resulted equations over the domain $\Omega$, we get
	\begin{equation}\label{pn1}
	g\int_{\Gamma}|\nn\zeta_{\alpha}|^2dx'+\int_{\Omega}\left[\p_tv_{\alpha}-\Delta v_{\alpha}+f\vec{\kappa}\times v_{\alpha}\right]\cdot\nn\zeta_{\alpha}\phi(x_3)dx=\int_{\Omega}D^{\alpha}G_1\cdot\nn\zeta_{\alpha}\phi(x_3)dx,
	\end{equation}
	where $V_{\alpha}=D^{\alpha}V=\p_t^{\alpha_0}\p_1^{\alpha_1}\p_2^{\alpha_2}V$ for any regular function $V$ with the multi index $\alpha=(\alpha_0,\alpha_1,\alpha_2,0)$ satisfying $|\alpha|\leq3$.
	By Lemma \ref{pln1}, we have	$$\left|\int_{\Omega}\left[\p_tv_{\alpha}+f\vec{\kappa}\times v_{\alpha}\right]\cdot\nn\zeta_{\alpha}\phi(x_3)dx\right|\lesssim \left[\|\p_tv_{\alpha}\|_0+\|v_{\alpha}\|_0\right]\cdot|\nn\zeta_{\alpha}|_0,$$
	\begin{align*}
	\left|-\int_{\Omega}\Delta v_{\alpha}\cdot\nn\zeta_{\alpha}\phi(x_3)dx\right|=&\left|-\int_{\Omega}\Delta_* v_{\alpha}\cdot\nn\zeta_{\alpha}\phi(x_3)dx+\int_{\Omega}\p_3v_{\alpha}\cdot\nn\zeta_{\alpha}\phi'(x_3)dx\right|\\
	\lesssim&|\nn\zeta_{\alpha}|_0\cdot\left[\|\Delta_*v_{\alpha}\|_0+\|\p_3v_{\alpha}\|_0\right]
	\end{align*}
	and
	$$\left|\int_{\Omega}D^{\alpha}G_1\cdot\nn\zeta_{\alpha}\phi(x_3)dx\right|
	\lesssim\|D^{\alpha}G_1\|_0\cdot|\nn\zeta_{\alpha}|_0
	\lesssim \left[\E(t)+\F(t)\right]\D(t)\cdot|\nn\zeta_{\alpha}|_0,$$
	and therefore combining these estimates with \eqref{pn1} together, we get
	\begin{equation}\label{pn2}
	 |\nn\zeta_{\alpha}|_0^2\lesssim\|\p_tv_{\alpha}\|_0^2+\|v_{\alpha}\|_0^2+\|\Delta_*v_{\alpha}\|_0^2+\|\p_3v_{\alpha}\|_0^2+\left[\E(t)+\F(t)\right]\D(t)^2
	\end{equation}
	for $\alpha=(\alpha_0,\alpha_1,\alpha_2,0)$ satisfying $|\alpha|\leq3$.

\par By the equation \eqref{pl3}$_1$, it holds that for any $\alpha=(\alpha_0,\alpha_1,\alpha_2,0)$
\begin{equation}\label{pn3}
\left\{\begin{aligned}
&-\p_3^2v_{\alpha}=D^{\alpha}G_1-\p_tv_{\alpha}-g\nn\zeta_{\alpha}+\Delta_*v_{\alpha}-f\vec{\kappa}\times v_{\alpha}\qquad &\text{in }\Omega,\\
&\p_{3}v_{\alpha}=D^{\alpha}G_3&\text{on }\Gamma,\\
&v_{\alpha}=0&\text{on }\Sigma_{b},
\end{aligned}\right.
\end{equation}
where the nonlinear term $G_3$ is defined by \eqref{C1}.

\par To get the estimate of $\|\p_{3}v_{\alpha}\|_0$ for $\alpha=(\alpha_0,\alpha_1,\alpha_2,0)$ satisfying $|\alpha|\leq3$, multiplying \eqref{pn3} by $v_{\alpha}$ and integrating the resulted equation by parts over $\Omega$, we have
$$\int_{\Omega}|\p_{3}v_{\alpha}|^2dx=\int_{\Gamma}D^{\alpha}G_3\cdot v_{\alpha}dx'+\int_{\Omega}\left[D^{\alpha}G_1-\p_tv_{\alpha}-g\nn\zeta_{\alpha}+\Delta_*v_{\alpha}\right]\cdot v_{\alpha}dx:=III_1+III_2.$$
Integrating by parts and combining Lemma \ref{pln1} together, we obtain
\begin{equation}\label{C2}
\begin{aligned}
|III_2|&=\int_{\Omega}\left[D^{\alpha-\beta}G_1-\p_tv_{\alpha-\beta}-g\nn\zeta_{\alpha-\beta}+\Delta_*v_{\alpha-\beta}\right]\cdot v_{\alpha+\beta}dx\\
&\lesssim \left[\|D^{\alpha-\beta}G_1\|_0+\|\p_tv_{\alpha-\beta}\|_0+|\nn\zeta_{\alpha-\beta}|_0+\|\Delta_*v_{\alpha-\beta}\|_0\right]\cdot\|v_{\alpha+\beta}\|_0\\
&\lesssim \E(t)^{3}+\sum_{|\alpha|\leq 4}\left(\|D^{\alpha}v\|_0^2+|D^{\alpha}\zeta|_0^2\right),
\end{aligned}
\end{equation}
where the multi index $\beta=(0,\beta_1,\beta_2,0)$ satisfies $\beta\leq\alpha$ and either $|\beta|=1$ if $\alpha_1+\alpha_2\geq1$ or $\beta=0$ if $\alpha_1+\alpha_2=0$.

\par The definition of $G_3$ \eqref{C1} gives
$$\begin{aligned}
D^{\alpha}G_3=&\sum_{\gamma<\alpha}C_{\alpha,\gamma}\left[D^{\alpha-\gamma}\nn\zeta \cdot \nn D^{\gamma}v-D^{\alpha-\gamma}(-|\nn\zeta|^2K+K\p_3\theta) \cdot \p_3D^{\gamma}v\right]\\&+\nabla_*\zeta\cdot\nabla_*v_{\alpha}+(-|\nabla_*\zeta|^2+\p_3\theta)K\p_{3}v_{\alpha},\end{aligned}$$
where $C_{\alpha,\gamma}>0$ is a constant only dependent of the multi indices $\alpha$ and $\gamma$. Then we get that for $|\alpha|\leq3$
\begin{equation*}
\begin{aligned}
&\left|\int_{\Gamma}\left[\nabla_*\zeta\cdot\nabla_*v_{\alpha}+(-|\nabla_*\zeta|^2+\p_3\theta)K\p_{3}v_{\alpha}\right]\cdot v_{\alpha}dx'\right|\\
\lesssim& |\nabla v_{\alpha}|_{-1/2}\cdot \left[|\nn\zeta v_{\alpha}|_{\frac{1}{2}}+\left|(-|\nn\zeta|^2+\p_3\theta)Kv_{\alpha}\right|_{\frac{1}{2}}\right]\\
\lesssim&\|\nabla v_{\alpha}\|_{0}\cdot\left[|\nn\zeta|_3+|\nn\zeta|_3^2+|\p_3\theta|_2\right]\cdot(1+\|\theta\|_4)\cdot|v_{\alpha}|_{\frac{1}{2}}\\
\lesssim&\E(t)^3
\end{aligned}
\end{equation*}
and
\begin{equation*}
\begin{aligned}
&\left|\int_{\Gamma}\sum_{\gamma<\alpha}C_{\alpha,\gamma}\left[D^{\alpha-\gamma}\nn\zeta \cdot \nn D^{\gamma}v-D^{\alpha-\gamma}(-|\nn\zeta|^2K+K\p_3\theta) \cdot \p_3D^{\gamma}v\right]\cdot v_{\alpha}dx'\right|\\
\lesssim&(|\zeta|_4+|\p_t\zeta|_2)\cdot(1+|\zeta|_4)^2\cdot(\|v\|_4^2+\|\p_tv\|_2^2)\\ \lesssim&\E(t)^3.
\end{aligned}
\end{equation*}
Adding above two estimates together, we obtain
$$|III_1|\lesssim\E(t)^{3},$$
which, combining with \eqref{C2} together, gives
\begin{equation}\label{C3}
\|\p_{3}v_{\alpha}\|_0\lesssim \E(t)^{3/2}+\sum_{|\alpha|\leq 6}\left(\|D^{\alpha}v\|_0+|D^{\alpha}\zeta|_0\right),\qquad\text{for }|\alpha|=|(\alpha_0,\alpha_1,\alpha_2,0)|\leq3.
\end{equation}

\par Next, we will use the iteration method to obtain the other estimates in \eqref{pnnn1} and \eqref{pnnn2}. By $\eqref{pn3}_1$ and Lemma \ref{pln1}, we have that for $\alpha=(\alpha_0,\alpha_1,\alpha_2,0)$ and $|\alpha|\leq 2$
\begin{equation}\label{pn4}
\begin{aligned}
\|\p_3^2v_{\alpha}\|_0&\leq \|D^{\alpha}G_1\|_0+\|\p_tv_{\alpha}\|_0+g\|\nn\zeta_{\alpha}\|_0+\|\Delta_*v_{\alpha}\|_0+f\|v_{\alpha}\|_0\\
&\lesssim \E(t)^2+\|\p_tv_{\alpha}\|_0+|\nn\zeta_{\alpha}|_0+\|\Delta_*v_{\alpha}\|_0+\|v_{\alpha}\|_0
\end{aligned}\end{equation}
and
\begin{equation}\label{pn5}
\begin{aligned}
\|\p_3^2v_{\alpha}\|_1&\leq \|D^{\alpha}G_1\|_1+\|\p_tv_{\alpha}\|_1+g\|\nn\zeta_{\alpha}\|_1+\|\Delta_*v_{\alpha}\|_1+f\|v_{\alpha}\|_1\\
&\lesssim \left[\E(t)+\F(t)\right]\D(t)+\|\p_tv_{\alpha}\|_1+|\nn\zeta_{\alpha}|_1+\|\Delta_*v_{\alpha}\|_1+\|v_{\alpha}\|_1.
\end{aligned}
\end{equation}
Differentiating $\eqref{pn3}_1$ with respect to the vertical direction $x_3$, we get that for the multi index $\alpha=(\alpha_0,\alpha_1,\alpha_2,0)$
\begin{equation}\label{pnn34}
-\p_3^3v_{\alpha}=\p_3D^{\alpha}G_1-\p_3\p_tv_{\alpha}+\Delta_*\p_{3}v_{\alpha}-f\vec{\kappa}\times \p_3v_{\alpha},
\end{equation}
and by Lemma \ref{pln1}, we have for $|\alpha|\leq 1$
\begin{equation}\label{pn6}
\begin{aligned}
\|\p_3^3v_{\alpha}\|_0&
\leq \|\p_3D^{\alpha}G_1\|_0+\|\p_3\p_tv_{\alpha}\|_0+\|\Delta_*\p_{3}v_{\alpha}\|_0+\|f\vec{\kappa}\times \p_3v_{\alpha}\|_0\\&
\lesssim \E(t)^2+\|\p_t\p_{3}v_{\alpha}\|_0+\|\Delta_*\p_{3}v_{\alpha}\|_0+\|\p_{3}v_{\alpha}\|_0
\end{aligned}\end{equation}
and
\begin{equation}\label{pn7}
\begin{aligned}
\|\p_3^3v_{\alpha}\|_1&
\leq \|\p_3D^{\alpha}G_1\|_1+\|\p_3\p_tv_{\alpha}\|_1+\|\Delta_*\p_{3}v_{\alpha}\|_1+\|f\vec{\kappa}\times \p_3v_{\alpha}\|_1\\&
\lesssim \left[\E(t)+\F(t)\right]\D(t)+\|\p_t\p_3v_{\alpha}\|_1+\|\Delta_*\p_3v_{\alpha}\|_1+\|\p_3v_{\alpha}\|_1.
\end{aligned}
\end{equation}
Note that the right side in \eqref{pn6} and \eqref{pn7} are bounded by the inequalities \eqref{pn4} and \eqref{pn5}, respectively.

\par Differentiating \eqref{pnn34} with respect to $x_3$ again, we get that
$$-\p_3^4v=\p_3^2G_1-\p_3^2\p_tv+\Delta_*\p^2_{3}v-f\vec{\kappa}\times \p_3^2v,$$
and by Lemma \ref{pln1}, we have
\begin{equation}\label{pn61}
\begin{aligned}
\|\p_3^4v\|_0&
\leq \|\p_3^2G_1\|_0+\|\p_3^2\p_tv\|_0+\|\Delta_*\p^2_{3}v\|_0+\|f\vec{\kappa}\times \p_3^2v\|_0\\&
\lesssim \E(t)^2+\|\p_3^2\p_tv\|_0+\|\Delta_*\p^2_{3}v\|_0+\|\p_3^2v\|_0
\end{aligned}\end{equation}
and
\begin{equation}\label{pn71}
\begin{aligned}
\|\p_3^4v\|_1&
\leq \|\p_3^2G_1\|_1+\|\p_3^2\p_tv\|_1+\|\Delta_*\p_{3}^2v\|_1+\|f\vec{\kappa}\times \p_3^2v\|_1\\&
\lesssim \left[\E(t)+\F(t)\right]\D(t)+\|\p_3^2\p_tv\|_1+\|\Delta_*\p_{3}^2v\|_1+\|\p_3^2v\|_1.
\end{aligned}
\end{equation}

\par Adding \eqref{pn2}, \eqref{C3}-\eqref{pn5} and \eqref{pn6}-\eqref{pn71} together, we obtain the inequality \eqref{pnnn1} and
\begin{equation}\label{pn9}
\sum_{i=0}^{1}\left(\|\p_t^iv\|_{5-2i}+|\nn\p_t^i\zeta|_{3-2i}\right)\lesssim \left(\E(t)^{1/2}+\F(t)^{1/2}\right)\D(t)+\sum_{\substack{\alpha=(\alpha_0,\alpha_1,\alpha_2,0)\\|\alpha|\leq 4}}\|D^{\alpha}v\|_1.
\end{equation}
\par To obtain the estimates of the free boundary $\zeta$, by the kinematic boundary condition \eqref{p4}$_3$, we get
\begin{equation}\label{shl10}
|\p_t\zeta|_{\frac{7}{2}}\lesssim\|w\|_{4}+|v\cdot\nn\zeta|_{\frac{7}{2}}\lesssim \|w\|_{4}+\|v\|_4\cdot|\nn\zeta|_{\frac{7}{2}},
\end{equation}
\begin{equation}\label{shl11}
|\p_t^2\zeta|_{\frac{3}{2}}\lesssim \|\p_tw\|_2+|\p_t(v\cdot\nn\zeta)|_{\frac{3}{2}}\lesssim \|\p_tw\|_2 +\|\p_tv\|_2\cdot|\nn\zeta|_{\frac{3}{2}}+\|v\|_2\cdot|\p_t\zeta|_{\frac{5}{2}}
\end{equation}
and
\begin{equation}\label{shl12}
\begin{aligned}
|\p_t^3\zeta|_{-\frac{1}{2}}&\lesssim \|\p_t^2w\|_0+|\p_t^2(v\cdot\nn\zeta)|_{-\frac{1}{2}}\\
&\lesssim\|\p_t^2w\|_0+|\p_t^2v|_{-\frac{1}{2}}\cdot|\nn\zeta|_{\frac{3}{2}}+|\p_tv|_{-\frac{1}{2}}\cdot|\nn\p_t\zeta|_{\frac{3}{2}} +|v|_{\frac{3}{2}}\cdot|\nn\p_t^2\zeta|_{-\frac{1}{2}}\\
&\lesssim\|\p_t^2w\|_0+\|\p_t^2v\|_0\cdot|\nn\zeta|_{\frac{3}{2}}+\|\p_tv\|_0\cdot|\p_t\zeta|_{\frac{5}{2}} +\|v\|_2\cdot|\p_t^2\zeta|_{\frac{1}{2}}.
\end{aligned}\end{equation}
By the relation \eqref{p5}, we have
\begin{equation*}\begin{aligned}
&\|D^{\alpha}w\|_0\leq\left\|-\int_{-b}^{x_3}D^{\alpha}(J\nna v)(t,x',s)ds\right\|_0\lesssim \sum_{i=0}^{2}\|\p_t^iv\|_{5-2i}+\left(\E(t)+\F(t)\right)\D(t),\qquad\text{for } \forall |\alpha|\leq 4,\\
&\|\p_{3}^kD^{\alpha}w\|_0\leq\left\|-\p_{3}^{k-1}D^{\alpha}(J\nna v)\right\|_0\lesssim \sum_{i=0}^{2}\|\p_t^iv\|_{4-2i}+\E(t)\D(t),\quad\text{for } \forall 1\leq k\leq 4\text{ and } |\alpha|\leq 4-k,
\end{aligned}
\end{equation*}
where the multi index $\alpha=(\alpha_0,\alpha_1,\alpha_2,0)$.
Combining the above two inequalities with \eqref{shl10}, \eqref{shl11} and \eqref{shl12} together, we obtain
\begin{equation}\label{pn10}
\sum_{i=1}^{3}|\p_t^i\zeta|_{\frac{11}{2}-2i}\lesssim\sum_{i=0}^{2}\|\p_t^{i}v\|_{5-2i}+\left(\E(t)+\F(t)\right)\D(t).
\end{equation}
Adding \eqref{pn9} and \eqref{pn10} together, we get the estimate \eqref{pnnn2}. Then the proof is completed.
\end{proof}

\subsection{Estimates of the free surface}
\par In this subsection, we will establish the estimates on the free surface $\zeta$, i.e. the norm of $|\nn\zeta|_{\frac{7}{2}}$, which is used to control the nonliear terms in the tangential and normal estimates. Unlike the incompressible Navier-Stokes equations, the kinematic boundary condition \eqref{p61} is not enough for us to obtain the bound of $|\nn\zeta|_{\frac{7}{2}}$ because the regularity of the vertical velocity $w$ is one order lower than the horizontal velocity $v$. One important observation is the relation between the free surface $\zeta$ and the pressure $P$, describing in \eqref{p5}, which inspires us to combine the kinetic boundary condition \eqref{p61} and the horizontal momentum equations together. So integrating with respect to the vertical direction from the bottom to the top, we get the new equation of the free surface $\zeta$, which is the wave equation with the strong damping term. And then the expectant estimate of $\zeta$ will be established by this new equation.

\par We first derive the new equation which the free surface $\zeta$ satisfies. Define $$\varphi(t,x'):=w(t,x',0)=-\int_{-b}^{0}(J\nna\cdot v)(t,x',x_3)dx_3.$$
Applying the operator $J\nna\cdot$ to the horizontal momentum equations \eqref{p6}$_1$ and integrating the resulted equation with respect to the vertical direction over $[-b,0]$, we have
\begin{equation}\label{pf1}
-\p_t\varphi+\int_{-b}^0(gJ\Delta_*\zeta)dx_3-\int_{-b}^0\left(\Delta_{\mathcal{A}}(J\nna\cdot v)\right)dx_3+\int_{-b}^0f\left(J(-\bar{\p_1}v^2+\bar{\p_2}v^1)\right)dx_3
=\Phi_1,
\end{equation}
where \begin{equation}\label{pfn1}
\begin{aligned}
\Phi_1:=&\int_{-b}^0\bigg[J\nna\cdot\left(\p_t\theta K\p_3v-v\cdot\nna v-wK\p_3v\right)+\p_t(J\nna\cdot v)-J\nna\cdot\p_tv\\&\qquad+J\nna\cdot\Delta_{\mathcal{A}}v-\Delta_{\mathcal{A}}(J\nna\cdot v)\bigg]dx_3.
\end{aligned}\end{equation}
Then it holds
\begin{align*}
&\int_{-b}^0(gJ\Delta_*\zeta)dx_3=g\Delta_*\zeta\cdot\int_{-b}^0(1+\p_3\theta)dx_3=g\Delta_*\zeta\cdot(b+\zeta),\\
&-\int_{-b}^0\left(\Delta_{\mathcal{A}}(J\nna\cdot v)\right)dx_3=-\int_{-b}^0\left[\Delta(J\nna\cdot v)+(\Delta_{\mathcal{A}}-\Delta)(J\nna\cdot v)\right]dx_3\\
&\qquad\qquad\qquad\qquad\qquad\quad=\Delta_*\varphi-\p_3(J\nna\cdot v)\arrowvert_{x_3=-b}^{0}-\int_{-b}^0(\Delta_{\mathcal{A}}-\Delta)(J\nna\cdot v)dx_3,
\end{align*}
and combining the above equalities with \eqref{pf1} together, we get
\begin{equation}\label{pf2}
-\p_t\varphi+gb\Delta_*\zeta+\Delta_*\varphi=\Phi_2,
\end{equation}
where \begin{equation}\label{pfn2}
\begin{aligned}
\Phi_2:=&-\int_{-b}^0f\left(J(-\bar{\p_1}v^2+\bar{\p_2}v^1)\right)dx_3+\p_3(J\nna\cdot v)\arrowvert_{x_3=-b}^{0}\\&+\int_{-b}^0(\Delta_{\mathcal{A}}-\Delta)(J\nna\cdot v)dx_3-g\zeta\Delta_*\zeta+\Phi_1.\end{aligned}\end{equation}
Combining \eqref{pf2} and the kinetic boundary condition \eqref{p61} together, we obtain
\begin{equation}\label{nfe}
\p_t^2\zeta-gb\Delta_*\zeta-\Delta_*\p_t\zeta=\Phi\qquad \text{on }\Gamma,
\end{equation}
where $\Gamma:=\T^2\text{ or }\R^2$ and
\begin{equation}\label{pfn3}
\Phi:=-\Phi_2-\p_t(v\cdot\nn\zeta)+\Delta_*(v\cdot\nn\zeta),
\end{equation}
where $\Phi_i$ ($i=1,2$) is defined by \eqref{pfn1} and \eqref{pfn2}.

\par Define $\eta:=(1+|\nn|)^{\frac{3}{2}}\nn\zeta$, satisfying the equation
\begin{equation}\label{pf3}
\p_t^2\eta-gb\Delta_*\eta-\Delta_*\p_t\eta=(1+|\nn|)^{\frac{3}{2}}\nn\Phi.
\end{equation}

\par To give the estimate of $\zeta$, the following commutator estimate is needed to bound the nonlinear terms. Define the commutator by
 $$\left[(1+|\nn|)^{\frac{3}{2}}\nn, v\cdot\nn\right]f:=(1+|\nn|)^{\frac{3}{2}}\nn\left(v\cdot\nn f\right)-v\cdot\nn\left((1+|\nn|)^{\frac{3}{2}}\nn f \right).$$
\begin{lemma}\label{CE}
	Let $v(x',x_3)\in H^3(\Omega)$ and $\nn f\in H^{\frac{3}{2}}(\Gamma)$. Then we have
	$$\left|\left[(1+|\nn|)^{\frac{3}{2}}\nn, v\cdot\nn\right]f\right|_0\lesssim \|v\|_3|\nn f|_{\frac{3}{2}}.$$
\end{lemma}
\begin{proof}
	Without loss of generity, we only give the proof in the horizontal infinite space. So the commutator $\Psi_1:=\left[(1+|\nn|)^{\frac{3}{2}}\nn, v\cdot\nn\right]f$ can be rewritten by the Fourier transformation
	\begin{equation*}
	\begin{aligned}
	 \hat{\Psi}_1(\xi)=&i\xi(1+|\xi|)^{\frac{3}{2}}\int_{\R^2}\hat{v}(\xi-s,0)\cdot(is\hat{f}(s))ds\\&-\int_{\R^2}\hat{v}(\xi-s,0)\cdot(is)\left((1+|s|)^{\frac{3}{2}}(is)\hat{f}(s)\right)ds\\
	=&-\int_{\R^2}s\cdot\hat{v}(\xi-s,0)\hat{f}(s)\left[\xi(1+|\xi|)^{\frac{3}{2}}-s(1+|s|)^{\frac{3}{2}}\right]ds.
	\end{aligned}\end{equation*}
	We claim that
	$$\left|\xi(1+|\xi|)^{\frac{3}{2}}-s(1+|s|)^{\frac{3}{2}}\right|\lesssim |\xi-s|\left[(1+|s|)^{\frac{3}{2}}+(1+|\xi-s|)^{\frac{3}{2}}\right].$$
	Indeed, this claim is showed by direct calculation
	\begin{align*}
	 &\xi(1+|\xi|)^{\frac{3}{2}}-s(1+|s|)^{\frac{3}{2}}\\=&(\xi-s)(1+|\xi|)^{\frac{3}{2}}+s\left[(1+|\xi|)^{\frac{3}{2}}-(1+|s|^{\frac{3}{2}})\right]\\
	 =&(\xi-s)(1+|\xi|)^{\frac{3}{2}}+s(|\xi|-|s|)\left[2+|\xi|+|s|+(1+|\xi|)^{\frac{1}{2}}(1+|s|)^{\frac{1}{2}}\right]\cdot\left[(1+|s|)^{\frac{1}{2}}+(1+|\xi|)^{\frac{1}{2}}\right]^{-1}.
	\end{align*}
	Therefore by the Cauchy inequality, we get this claim.
	
	\par By Young inequality and Planchel theorem, we obtain
	\begin{equation*}
	\begin{aligned}
	|\Psi_1|_0=|\hat{\Psi}_1|_0\lesssim& \left|-\int_{\R^2}\left(|\xi-s||\hat{v}(\xi-s,0)|\right)\cdot\left((1+|s|)^{\frac{3}{2}}|s\hat{f}|\right)ds\right|_0\\&+\left|\int_{\R^2}\left(|\xi-s|(1+|\xi-s|)^{\frac{3}{2}}|\hat{v}(\xi-s,0)|\right)\cdot\left(|s\hat{f}|\right)ds\right|_0\\
	 \lesssim&\left|(1+|\xi|)^{\frac{3}{2}}|\xi\hat{f}|\right|_0\cdot\int_{\R^2}(1+|\xi|)^{-\frac{3}{2}}\cdot\left((1+|\xi|)^{\frac{3}{2}}|\xi||\hat{v}(\xi,0)|\right)d\xi\\&+\left||\xi|(1+|\xi|)^{\frac{3}{2}}|\hat{v}(\xi,0)|\right|_0\cdot\int_{\R^2}(1+|\xi|)^{-\frac{3}{2}}\cdot|(1+|\xi|)^{\frac{3}{2}}\xi\hat{f}(\xi)|d\xi\\
	\lesssim & \|v\|_3|\nn f|_{\frac{3}{2}}.
	\end{aligned}
	\end{equation*}
	Note that H\"older inequality is used in the last inequality.
\end{proof}

By Lemma \ref{CE}, we have the following estimate on the free boundary $\zeta$.
\begin{proposition}\label{FBE}
	Let $T>0$ and $(v,w,\zeta)$ be the strong solution to equations \eqref{p6}-\eqref{p61}. Suppose the assumptions in Theorem \ref{th1} or Theorem \ref{th2} hold. Then, under the a-priori assumption \eqref{a-priori assumption}, we have
	\begin{equation}\label{pf}
	\begin{aligned}
	 &\frac{d}{dt}\int_{\Gamma}\left[|\p_t\eta|^2-\p_t\eta\cdot\Delta_*\eta+|\Delta_*\eta|^2+gb|\nn\eta|^2\right]dx'+\int_{\Gamma}\left[|\nn\p_t\eta|^2+|\Delta_*\eta|^2\right]dx'\\\lesssim & (1+\E(t)+\F(t))\D(t)^2,
	\end{aligned}
	\end{equation}
	where the functions $\E(t)$, $\D(t)$ and $\F(t)$ are defined by \eqref{DED}.
\end{proposition}
\begin{proof}
	To obtain this estimate on the free boundary $\zeta$,
	multiplying \eqref{pf3} by $\p_t\eta$ and integrating the resulted equation by parts over $\Gamma$, we have
	\begin{equation}\label{pf4}
	 \frac{1}{2}\frac{d}{dt}\int_{\Gamma}\left[|\p_t\eta|^2+gb|\nn\eta|^2\right]dx'+\int_{\Gamma}|\nn\p_t\eta|^2dx'=\int_{\Gamma}(1+|\nn|)^{\frac{3}{2}}\nn\Phi\cdot\p_t\eta dx'.
	\end{equation}
	And then multiplying \eqref{pf3} by $-\Delta_*\eta$ and integrateing the resulted equation by parts over $\Gamma$, we get
	\begin{equation}\label{pf5}
	\begin{aligned}
	 &\frac{d}{dt}\int_{\Gamma}\left[-\p_t\eta\cdot\Delta_*\eta+|\Delta_*\eta|^2\right]dx'-\int_{\Gamma}|\nn\p_t\eta|^2dx'+gb\int_{\Gamma}|\Delta_*\eta|^2dx'\\=&\int_{\Gamma}(1+|\nn|)^{\frac{3}{2}}\nn\Phi\cdot(-\Delta_*\eta)dx'.
	\end{aligned}
	\end{equation}
	Therefore adding \eqref{pf4} and \eqref{pf5} together, we get
	\begin{equation} \label{pf6}
	\begin{aligned}
	 &\frac{d}{dt}\int_{\Gamma}\left[|\p_t\eta|^2-\p_t\eta\cdot\Delta_*\eta+|\Delta_*\eta|^2+gb|\nn\eta|^2\right]dx'+\int_{\Gamma}\left[|\nn\p_t\eta|^2+gb|\Delta_*\eta|^2\right]dx'\\=&\int_{\Gamma}\left[\p_t\eta-\Delta_*\eta\right]\cdot(1+|\nn|)^{\frac{3}{2}}\nn\Phi dx'\\
	:=&\sum\limits_{i=1}^6\Rmnum{4}_i,
	\end{aligned}\end{equation}
where the nonlinear terms $\Rmnum{4}_i$ are defined by
\begin{equation}\label{NNNNN}
\begin{aligned}
&\Rmnum{4}_1=\int_{\Gamma}\left\{\left[\p_t\eta-\Delta_*\eta\right]\cdot(1+|\nn|)^{\frac{3}{2}}\nn(-\p_t(v\cdot\nn\zeta))\right\}dx',\\
&\Rmnum{4}_2=\int_{\Gamma}\left\{\left[\p_t\eta-\Delta_*\eta\right]\cdot(1+|\nn|)^{\frac{3}{2}}\nn(\Delta_*(v\cdot\nn\zeta))\right\}dx',\\
&\Rmnum{4}_3=\int_{\Gamma}\Bigg\{\left[\p_t\eta-\Delta_*\eta\right]\cdot(1+|\nn|)^{\frac{3}{2}}\nn\Bigg[\p_3(J\nna\cdot v)\arrowvert_{x_3=-b}^{0}-\int_{-b}^0f\left(J(-\bar{\p_1}v^2+\bar{\p_2}v^1)\right)dx_3\\
&\qquad\qquad\qquad-g\zeta\Delta_*\zeta+\int_{-b}^0\left[J\nna\cdot\left(\p_t\theta K\p_3v-v\cdot\nna v-wK\p_3v\right)\right]dx_3\Bigg]\Bigg\}dx',\\
&\Rmnum{4}_4=\int_{\Gamma}\left\{\left[\p_t\eta-\Delta_*\eta\right]\cdot(1+|\nn|)^{\frac{3}{2}}\nn\int_{-b}^0\left[(\Delta_{\mathcal{A}}-\Delta)(J\nna\cdot v)\right]dx_3\right\}dx',\\
&\Rmnum{4}_5=\int_{\Gamma}\left\{\left[\p_t\eta-\Delta_*\eta\right]\cdot(1+|\nn|)^{\frac{3}{2}}\nn\int_{-b}^0\left[\p_t(J\nna\cdot v)-J\nna\cdot\p_tv\right]dx_3\right\}dx',\\
&\Rmnum{4}_6=\int_{\Gamma}\left\{\left[\p_t\eta-\Delta_*\eta\right]\cdot(1+|\nn|)^{\frac{3}{2}}\nn\int_{-b}^0\left[J\nna\cdot\Delta_{\mathcal{A}}v-\Delta_{\mathcal{A}}(J\nna\cdot v)\right]dx_3\right\}dx'.
\end{aligned}
\end{equation}
\par These nonlinear terms can be estimated below one by one.

\par Indeed, by Lemma \ref{SME}, we have
\begin{equation}\label{pf7}
\begin{aligned}
|\Rmnum{4}_1|=&\left|\int_{\Gamma}\nn\p_t\eta\cdot(1+|\nn|)^{\frac{3}{2}}(-\p_t(v\cdot\nn\zeta))dx'+\int_{\Gamma}\Delta_*\eta\cdot(1+|\nn|)^{\frac{3}{2}}\nn(-\p_t(v\cdot\nn\zeta))dx'\right|\\
\lesssim& \left[|\p_t\nn\eta|_0+|\Delta_*\eta|_0\right]\cdot\left[|\p_tv\cdot\nn\zeta|_{\frac{5}{2}}+|v\cdot\nn\p_t\zeta|_{\frac{5}{2}}\right]\\
\lesssim& \left[|\p_t\nn\eta|_0+|\Delta_*\eta|_0\right]\cdot\left[\|\p_tv\|_3|\nn\zeta|_{\frac{5}{2}}+\|v\|_3|\p_t\zeta|_{\frac{7}{2}}\right]\\
\lesssim&\left[|\p_t\nn\eta|_0+|\Delta_*\eta|_0\right]\cdot\E(t)\D(t).
\end{aligned}
\end{equation}
To bound the nonlinear term $\Rmnum{4}_2$, by Leibniz rule, it holds
$$\Delta_*(v\cdot\nn\zeta)=v\cdot\nn\Delta_*\zeta+\left(2\sum_{i,j=1}^{2}\p_iv^j\p_j\p_i\zeta+\Delta_*v\cdot \nn\zeta\right):=\Rmnum{5}_1+\Rmnum{5}_2,$$
and then we get
\begin{equation}\label{pf8}
\begin{aligned}
&\left|\int_{\Gamma}\left[\p_t\eta-\Delta_*\eta\right]\cdot(1+|\nn|)^{\frac{3}{2}}\nn \Rmnum{5}_2dx'\right|\\
\lesssim&\left[|\p_t\nn\eta|_0+|\Delta_*\eta|_0\right]\cdot\left(|v|_{\frac{7}{2}}\cdot|\nn\zeta|_{\frac{7}{2}}+|v|_{\frac{9}{2}}\cdot|\nn\zeta|_{\frac{7}{2}}\right)\\
\lesssim&\left[|\p_t\nn\eta|_0+|\Delta_*\eta|_0\right]\cdot\left(\E(t)+\F(t)\right)\D(t).
\end{aligned}
\end{equation}
Integrating by parts, we get
\begin{equation*}
\begin{aligned}
&\int_{\Gamma}\left[\p_t\eta-\Delta_*\eta\right]\cdot(1+|\nn|)^{\frac{3}{2}}\nn(v\cdot\nn\Delta_*\zeta)dx'\\
=&\int_{\Gamma}\left[\p_t\eta-\Delta_*\eta\right]\cdot\left[v\cdot\nn(\Delta_*\eta)\right]dx'+\int_{\Gamma}\left[\p_t\eta-\Delta_*\eta\right]\cdot\left\{\left[(1+|\nn|)^{\frac{3}{2}}\nn, v\cdot\nn\right]\Delta_*\zeta\right\}dx'\\
=&\int_{\Gamma}\left[-\nn\cdot(\p_t\eta v)\Delta_*\eta+\frac{1}{2}\nn\cdot v|\Delta_*\eta|^2\right]dx'
+\int_{\Gamma}\left[\p_t\eta-\Delta_*\eta\right]\cdot\left\{\left[(1+|\nn|)^{\frac{3}{2}}\nn, v\cdot\nn\right]\Delta_*\zeta\right\}dx'\\
:=&\Rmnum{5}_3+\Rmnum{5}_4,
\end{aligned}\end{equation*}
and then
\begin{equation}\label{jj1}
|\Rmnum{5}_3|\lesssim \|v\|_4\left(|\p_t\eta|_1|\Delta_*\eta|_0+|\Delta_*\eta|_0^2\right)\lesssim \E(t)\left(|\nn\p_t\eta|_0^2+|\Delta_*\eta|_0^2\right)+\E(t)\D(t)|\Delta_*\eta|_0.
\end{equation}
By Lemma \ref{CE}, we have
\begin{equation}\label{jj2}
\begin{aligned}
|\Rmnum{5}_4|\lesssim&\left[|\p_t\eta|_0+|\Delta_*\eta|_0\right]\cdot\left|\left[(1+|\nn|)^{\frac{3}{2}}\nn, v\cdot\nn\right]\Delta_*\zeta\right|_0\\
\lesssim&\left[|\p_t\eta|_0+|\Delta_*\eta|_0\right]\cdot\|v\|_3\cdot|\nn\Delta_*\zeta|_{\frac{3}{2}}\\
\lesssim&\E(t)\D(t)|\Delta_*\eta|_0+\E(t)|\Delta_*\eta|_0^2.
\end{aligned}
\end{equation}
Adding \eqref{pf8}, \eqref{jj1} and \eqref{jj2} together, we obtain
\begin{equation}\label{pf9}
\begin{aligned}
\left|\Rmnum{4}_2\right|\lesssim \E(t)\left(|\nn\p_t\eta|_0^2+|\Delta_*\eta|_0^2\right)+\left[|\p_t\nn\eta|_0+|\Delta_*\eta|_0\right]\cdot\left(\E(t)+\F(t)\right)\D(t).
\end{aligned}
\end{equation}

\par We directly estimate the nonlinear term $\Rmnum{4}_3$ by Sobolev embedding inequalities and Lemma \ref{SME}, getting
\begin{equation}\label{pf10}
\begin{aligned}
\left|\Rmnum{4}_3\right|
\lesssim&\left[|\p_t\nn\eta|_0+|\Delta_*\eta|_0\right]\cdot\big[\|\p_3(J\nna\cdot v)\|_3+\|J(-\bar{\p_1}v^2+\bar{\p_2}v^1)\|_3+|\zeta\Delta_*\zeta|_{\frac{5}{2}}\\
&\qquad\qquad\qquad\qquad+\left\|J\nna\cdot\left(\p_t\theta K\p_3v-v\cdot\nna v-wK\p_3v\right)\right\|_3\big]\\
\lesssim&\left[|\p_t\nn\eta|_0+|\Delta_*\eta|_0\right]\cdot\left(1+\E(t)+\F(t)\right)\D(t).
\end{aligned}
\end{equation}

\par To estimate $\Rmnum{4}_4$, we calculate this term by the equality \eqref{t1}
$$\begin{aligned}
&\int_{-b}^0\left[(\Delta_{\mathcal{A}}-\Delta)(J\nna\cdot v)\right]dx_3\\=&\int_{-b}^0\Big[-\p_1(AK)\p_3\psi+AK\p_3(AK)\p_3\psi-\p_2(BK)\p_3\psi+BK\p_3(BK)\p_3\psi-K^3\p_3^2\theta\p_3\psi\\
&\quad-2AK\p_1\p_3\psi+A^2K^2\p_3^2\psi-2BK\p_2\p_3\psi+B^2K^2\p_3^2\psi-K^2(2\p_3\theta+|\p_3\theta|^2)\p_3^2\psi\Big]dx_3\\
=&\int_{-b}^0\Big[-\p_1(AK)\p_3\psi+AK\p_3(AK)\p_3\psi-\p_2(BK)\p_3\psi+BK\p_3(BK)\p_3\psi-K^3\p_3^2\theta\p_3\psi\\
&\quad+2\p_3(AK)\p_1\psi-\p_3(A^2K^2)\p_3\psi+2\p_3(BK)\p_2\psi-\p_3(B^2K^2)\p_3\psi+\p_3[K^2(2\p_3\theta+|\p_3\theta|^2)]\p_3\psi\Big]dx_3\\
&\quad+\left[-2AK\p_1\psi+A^2K^2\p_3\psi-2BK\p_2\psi+B^2K^2\p_3\psi-K^2(2\p_3\theta+|\p_3\theta|^2)\p_3\psi\right]\mid_{x_3=-b}^{0},
\end{aligned}$$
where $\psi:=J\nna\cdot v$. Then by Sobolev embedding inequalities and the generalized Minkowski inequality, we have
\begin{align*}
&\left|\int_{-b}^0\left[(\Delta_{\mathcal{A}}-\Delta)(J\nna\cdot v)\right]dx_3\right|_{\frac{5}{2}}\\
\lesssim&\Big\|-\p_1(AK)\p_3\psi+AK\p_3(AK)\p_3\psi-\p_2(BK)\p_3\psi+BK\p_3(BK)\p_3\psi-K^3\p_3^2\theta\p_3\psi\\
&+2\p_3(AK)\p_1\psi-\p_3(A^2K^2)\p_3\psi+2\p_3(BK)\p_2\psi-\p_3(B^2K^2)\p_3\psi+\p_3[K^2(2\p_3\theta+|\p_3\theta|^2)]\p_3\psi\Big\|_3\\
&+\left\|-2AK\p_1\psi+A^2K^2\p_3\psi-2BK\p_2\psi+B^2K^2\p_3\psi-K^2(2\p_3\theta+|\p_3\theta|^2)\p_3\psi\right\|_3\\
\lesssim&(\E(t)+\F(t))\D(t),
\end{align*}
and therefore we immediately get the estimate of $\Rmnum{4}_4$ by its definition in \eqref{NNNNN}
\begin{equation}\label{pf12}
\begin{aligned}
|\Rmnum{4}_4|\lesssim&\left[|\p_t\nn\eta|_0+|\Delta_*\eta|_0\right]\cdot\left|\int_{-b}^0\left[(\Delta_{\mathcal{A}}-\Delta)(J\nna\cdot v)\right]dx_3\right|_{\frac{5}{2}}\\
\lesssim&\left[|\p_t\nn\eta|_0+|\Delta_*\eta|_0\right]\cdot (\E(t)+\F(t))\D(t).
\end{aligned}
\end{equation}

\par To bound $\Rmnum{4}_5$, by direct computation, we get
$$\p_t(J\nna\cdot v)-J\nna\cdot\p_tv=\p_tJ\nna\cdot v+J(\p_t\a\n)_*\cdot v,$$
and then
\begin{equation}\label{pf14}
\begin{aligned}
\left|\Rmnum{4}_5\right|
\lesssim&\left[|\p_t\nn\eta|_0+|\Delta_*\eta|_0\right]\cdot\left[\|\p_tJ\nna\cdot v\|_3+\|J(\p_t\a\n)_*\cdot v\|_3\right]\\
\lesssim&\left[|\p_t\nn\eta|_0+|\Delta_*\eta|_0\right]\cdot (\E(t)+\F(t))\D(t).
\end{aligned}
\end{equation}

\par To estimate $\Rmnum{4}_6$, we get by the direct computation
$$J\nna\cdot\Delta_{\mathcal{A}}v-\Delta_{\mathcal{A}}(J\nna\cdot v)=-2\na J\cdot\na(\nna\cdot v)-\Delta_{\mathcal{A}}J\nna\cdot v,$$
and $$\begin{aligned}
\Delta_{\mathcal{A}}J=&\p_3\left[\p_1^2\theta-2AK\p_3\p_1\theta+\p_2^2\theta-2BK\p_3\p_2\theta-(A^2+B^2-1)K^2\p_3^2\theta\right]\\
&+\big[-\nn\cdot(\nn\theta K)\p_3J-(AK\p_3(AK)+BK\p_3(BK)-K\p_3K)\p_3J\\&\qquad+2\p_3(\nn\theta K)\nn\p_3\theta+\p_3\left((A^2+B^2-1)K^2\right)\p_3^2\theta\big]\\
:=&\Rmnum{5}_5+\Rmnum{5}_6.
\end{aligned}$$
Integrating by parts, we have
\begin{equation*}
\begin{aligned}
&\int_{-b}^0\Rmnum{5}_5\nna\cdot vdx_3\\
=&\left\{\left[\p_1^2\theta-2AK\p_3\p_1\theta+\p_2^2\theta-2BK\p_3\p_2\theta-(A^2+B^2-1)K^2\p_3^2\theta\right]\nna\cdot v\right\}\mid_{x_3=-b}^0\\
&-\int_{-b}^{0}\left\{\left[\p_1^2\theta-2AK\p_3\p_1\theta+\p_2^2\theta-2BK\p_3\p_2\theta-(A^2+B^2-1)K^2\p_3^2\theta\right]\p_3\left(\nna\cdot v\right)\right\}dx_3,
\end{aligned}\end{equation*}
which is bounded by
\begin{equation}\label{ttt1}
\begin{aligned}
&\left|\int_{\Gamma}\left\{\left[\p_t\eta-\Delta_*\eta\right]\cdot(1+|\nn|)^{\frac{3}{2}}\nn\int_{-b}^0\Rmnum{5}_5\nna\cdot vdx_3\right\}dx'\right|\\
\lesssim&\left[|\p_t\nn\eta|_0+|\Delta_*\eta|_0\right]\cdot\bigg[\left\|\left[\p_1^2\theta-2AK\p_3\p_1\theta+\p_2^2\theta-2BK\p_3\p_2\theta-(A^2+B^2-1)K^2\p_3^2\theta\right]\nna\cdot v\right\|_3\\
&+\|\left[\p_1^2\theta-2AK\p_3\p_1\theta+\p_2^2\theta-2BK\p_3\p_2\theta-(A^2+B^2-1)K^2\p_3^2\theta\right]\p_3(\nna\cdot v)\|_3\bigg]\\
\lesssim&\left[|\p_t\nn\eta|_0+|\Delta_*\eta|_0\right]\cdot (\E(t)+\F(t))\D(t).
\end{aligned}
\end{equation}
The remaining terms in $\Rmnum{4}_6$ are estimated directly by Sobolev emmbedding inequalities
\begin{equation}\label{ttt2}
\begin{aligned}
&\left|\int_{\Gamma}\left\{\left[\p_t\eta-\Delta_*\eta\right]\cdot(1+|\nn|)^{\frac{3}{2}}\nn\int_{-b}^0\left[-2\na J\cdot\na(\nna\cdot v)-\Rmnum{5}_6\nna\cdot v\right]dx_3\right\}dx'\right|\\
\lesssim&\left[|\p_t\nn\eta|_0+|\Delta_*\eta|_0\right]\cdot\left\|-2\na J\cdot\na(\nna\cdot v)-\Rmnum{5}_6\nna\cdot v\right\|_3\\
\lesssim&\left[|\p_t\nn\eta|_0+|\Delta_*\eta|_0\right]\cdot (\E(t)+\F(t))\D(t).
\end{aligned}
\end{equation}
Therefore we obtain the estimate of $\Rmnum{4}_6$ by adding \eqref{ttt1} and \eqref{ttt2} together
\begin{equation}\label{pf15}
\begin{aligned}
\left|\Rmnum{4}_6\right|
\lesssim \left[|\p_t\nn\eta|_0+|\Delta_*\eta|_0\right]\cdot (\E(t)+\F(t))\D(t).
\end{aligned}
\end{equation}

\par In summary, adding \eqref{pf7}, \eqref{pf9}-\eqref{pf14} and \eqref{pf15} together, we conclude that
\begin{equation*}
\begin{aligned}
&\frac{d}{dt}\int_{\Gamma}\left[|\p_t\eta|^2-\p_t\eta\cdot\Delta_*\eta+|\Delta_*\eta|^2+gb|\nn\eta|^2\right]dx'+\int_{\Gamma}\left[|\nn\p_t\eta|^2+|\Delta_*\eta|^2\right]dx'\\\lesssim &\left[|\p_t\nn\eta|_0+|\Delta_*\eta|_0\right]\cdot (1+\E(t)+\F(t))\D(t),
\end{aligned}
\end{equation*}
which completes the proof of the estimates \eqref{pf} by Cauchy inequality.
\end{proof}

\subsection{Proof of Theorem \ref{a-priori}}
\par In this subsection, we will establish the a-priori estimates by combining the temporal, tangential and normal estimates obtained in Subsection $2.1-2.4$ together.
\begin{proof}[\textbf{Proof of Theorem \ref{a-priori}.}]
	To establish the a-priori estimates, combining the temporal estimates in Proposition \ref{TE}, tangential estimates in Proposition \ref{Tang E} and normal estimates in Proposition \ref{NE} together and choosing $\delta$ small enough, under the a-priori assumption \eqref{a-priori assumption}, we obtain
	\begin{equation}\label{ppp}
	\frac{d}{dt}\E^2(t)+\vartheta\D^2(t)\leq0,
	\end{equation}
	where the energy $\E(t)$ and dissipation $\D(t)$ are defined by \eqref{DED} and the constant $\vartheta>0$ is independent of the time $T$. Therefore integrating \eqref{ppp} with respect to the time $t$ over $[0,T]$, we have
	\begin{equation}\label{ppp1}
	\sup_{0\leq t\leq T}\E^2(t)+\vartheta\int_{0}^{T}\D^2(t)dt\leq \E^2(0).
	\end{equation}
	By the estimates \eqref{pf} on the free boundary $\zeta$ in Proposition \ref{FBE}, we get
	\begin{equation}\label{aaa1}
	\begin{aligned}
	&\sup_{0\leq t\leq T}\int_{\Gamma}\left[|\p_t\eta|^2-\p_t\eta\cdot\Delta_*\eta+|\Delta_*\eta|^2+gb|\nn\eta|^2\right]dx'+\int_{0}^{T}\int_{\Gamma}\left[|\nn\p_t\eta|^2+|\Delta_*\eta|^2\right]dx'dt\\\lesssim &\int_{\Gamma}\left[|\p_t\eta|^2-\p_t\eta\cdot\Delta_*\eta+|\Delta_*\eta|^2+gb|\nn\eta|^2\right](0,x')dx'+\int_{0}^{T}(1+\E(t)+\F(t))\D(t)^2dt.
	\end{aligned}
	\end{equation}
	Note that
	$$\p_t\eta(0,x')=(1+|\nn|)^{\frac{3}{2}}\nn\left[w(0,x',0)-v(0,x',0)\cdot\nn\eta(0,x')\right]$$
	and by \eqref{p5}, \eqref{ppp1}, \eqref{aaa1} and the a-priori assumption \eqref{a-priori assumption}, we obtain
	$$\sup_{0\leq t\leq T}\int_{\Gamma}\left[|\p_t\eta|^2-\p_t\eta\cdot\Delta_*\eta+|\Delta_*\eta|^2+gb|\nn\eta|^2\right]dx'+\int_{0}^{T}\int_{\Gamma}\left[|\nn\p_t\eta|^2+|\Delta_*\eta|^2\right]dx'dt\lesssim \E^2(0)+\F^2(0),$$
	so combining this estimate with \eqref{ppp} and \eqref{ppp1} together, we completes the proof of Theorem \ref{a-priori}.
\end{proof}

\section{Proof of Theorem \ref{th2}}
\par In this section, we will show the proof of Theorem \ref{th2} in the horizontal infinite domain. The a-priori estimates in Theorem \ref{a-priori} also hold in this case, so there exists a unique global solution to equations \eqref{p6}-\eqref{p61} under the assumption of Theorem \ref{th2}. However, the decay rate can not be directly obtained by the a-priori estimates because of the invalidity of the Poincar\'e inequality on the free boundary $\zeta$. Here we adopt the negative-index Sobolev space to overcome this difficulty.
\par We first give the definition of the negative-index Sobolev space.
\begin{definition}\nonumber
	For any $\gamma\in (0,1)$,\\
	\begin{equation*}
	\begin{aligned}
	&v_{-\gamma}(t,x',x_3):=|\nn|^{-\gamma}v(t,x',x_3)=\int_{\R^2}e^{2\pi ix'\cdot\xi}|\xi|^{-\gamma}\hat{v}(t,\xi,x_3)d\xi;\\
	&w_{-\gamma}(t,x',x_3):=|\nn|^{-\gamma}w(t,x',x_3)=\int_{\R^2}e^{2\pi ix'\cdot\xi}|\xi|^{-\gamma}\hat{w}(t,\xi,x_3)d\xi;\\
	&\zeta_{-\gamma}(t,x'):=|\nn|^{-\gamma}\zeta(t,x')=\int_{\R^2}e^{2\pi ix'\cdot\xi}|\xi|^{-\gamma}\hat{\zeta}(t,\xi)d\xi,
	\end{aligned}
	\end{equation*}
	where $\hat{\cdot}$ means the Fourier transformation in $\R^2$.
\end{definition}

\par By means of the negative-index Sobolev space, we give the proof of Theorem \ref{th2}.

\begin{proof}[\textbf{Proof of Theorem \ref{th2}.}]
	By the a-priori estimates in Theorem \ref{a-priori} and the local existence in Proposition \ref{local existence} in the appendix, the global well-posedness is obtained, so we only need to prove the large time behavior.
	
	\par Applying the operator $|\nn|^{-\gamma}$ to the equations \eqref{pl1} and \eqref{pl2}, we have
	\begin{equation}\label{q1}
	\left\{\begin{aligned}
	&\p_tv_{-\gamma}+g\nn\zeta_{-\gamma}-\Delta v_{-\gamma}+f\vec{\kappa}\times v_{-\gamma}=|\nn|^{-\gamma}G_1,\\
	&\nn\cdot v_{-\gamma}+\p_3w_{-\gamma}=|\nn|^{-\gamma}G_2
	\end{aligned}\right.
	\end{equation}
	in the domain $\Omega$ with the boundary conditions
	\begin{equation}\label{q2}
	\left\{\begin{aligned}
	&\p_3v_{-\gamma}=|\nn|^{-\gamma}G_3 \qquad& \text{on } \Gamma,\\
	&v_{-\gamma}=w_{-\gamma}=0\qquad& \text{on } \Sigma_b,\\
	&\p_t\zeta_{-\gamma}=w_{-\gamma}+|\nn|^{-\gamma}G_4 \qquad& \text{on } \Gamma.
	\end{aligned}\right.
	\end{equation}
	Multiplying $\eqref{q1}_1$ by $v_{-\gamma}$ and integrating the resulted equation by parts over $\Omega$, we get
	\begin{equation*}
	\frac{1}{2}\frac{d}{dt}\int_{\Omega}|v_{-\gamma}|^2dx-\int_{\Omega}g\zeta_{-\gamma}\nn\cdot v_{-\gamma}dx+\int_{\Omega}|\n v_{-\gamma}|^2dx-\int_{\Gamma}|\nn|^{-\gamma}G_3\cdot v_{-\gamma}dx'
	=\int_{\Omega}|\nn|^{-\gamma}G_1\cdot v_{-\gamma}dx.
	\end{equation*}
	By the boundary conditions $\eqref{q1}_2$ and $\eqref{q2}_3$, it holds
	\begin{equation*}
	\begin{aligned}
	-\int_{\Omega}g\zeta_{-\gamma}\nn\cdot v_{-\gamma}dx&=\int_{\Omega}g\zeta_{-\gamma}\cdot[\p_3w_{-\gamma}-|\nn|^{-\gamma}G_2]dx=\int_{\Gamma}g\zeta_{-\gamma}\cdot w_{-\gamma}dx'-\int_{\Omega}g\zeta_{-\gamma}\cdot|\nn|^{-\gamma}G_2dx\\
	&=\frac{g}{2}\frac{d}{dt}\int_{\Gamma}|\zeta_{-\gamma}|^2dx'-\int_{\Gamma}g\zeta_{-\gamma}\cdot|\nn|^{-\gamma}G_4dx' -\int_{\Omega}g\zeta_{-\gamma}\cdot|\nn|^{-\gamma}G_2dx.
	\end{aligned}\end{equation*}
	Adding above two equalities together, we obtain
	\begin{equation*}
	\begin{aligned}
	&\frac{1}{2}\frac{d}{dt}\left[\int_{\Omega}|v_{-\gamma}|^2dx+g\int_{\Gamma}|\zeta_{-\gamma}|^2dx'\right]+\int_{\Omega}|\n v_{-\gamma}|^2dx\\
	=&\int_{\Omega}\left(|\nn|^{-\gamma}G_1\cdot v_{-\gamma}+g\zeta_{-\gamma}\cdot|\nn|^{-\gamma}G_2\right)dx+\int_{\Gamma}\left(|\nn|^{-\gamma}G_3\cdot v_{-\gamma}+g\zeta_{-\gamma}\cdot|\nn|^{-\gamma}G_4\right)dx'\\
	 \lesssim&\|v_{-\gamma}\|_1\cdot\left(\left\||\nn|^{-\gamma}G_1\right\|_0+\left||\nn|^{-\gamma}G_3\right|_0\right)+|\zeta_{-\gamma}|_0\cdot\left(\left\||\nn|^{-\gamma}G_2\right\|_0+\left||\nn|^{-\gamma}G_4\right|_0\right),
	\end{aligned}
	\end{equation*}
and by Poincar\'e inequality, it obtain
\begin{equation}\label{q4}
\begin{aligned}
&\frac{d}{dt}\left[\int_{\Omega}|v_{-\gamma}|^2+g\int_{\Gamma}|\zeta_{-\gamma}|^2\right]+\int_{\Omega}|\n v_{-\gamma}|^2\\
\lesssim& \left\||\nn|^{-\gamma}G_1\right\|_0^2+\left||\nn|^{-\gamma}G_3\right|_0^2+|\zeta_{-\gamma}|_0\cdot\left(\left\||\nn|^{-\gamma}G_2\right\|_0+\left||\nn|^{-\gamma}G_4\right|_0\right)\\
\lesssim&\left(\E(t)+|\zeta_{-\gamma}|_0\right)\D^2(t).
\end{aligned}
\end{equation}
In the last inequality, Lemma \ref{LN1} and Lemma \ref{LN2} in the appendix are used to estimate the nonlinear terms.

\par Integrating \eqref{q4} over $[0,t]$ and combining the resulted inequality with the a-priori estimate \eqref{a-prior-est} together, we obtain
\begin{equation*}
\begin{aligned}
\int_{\Gamma}|\zeta_{-\gamma}(t,\cdot)|^2&\leq C\left(\E^2(0)+\F^2(0)+\|v_{-\gamma}(0,\cdot)\|_0^2+|\zeta_{-\gamma}(0,\cdot)|_0^2+\int_{0}^{t}|\zeta_{-\gamma}(s,\cdot)|_0\cdot\D^2(s)ds\right),
\end{aligned}
\end{equation*}
where the positive constant $C$ is independent of the time $t$. Define
$$\varUpsilon(t):=\E^2(0)+\F^2(0)+\|v_{-\gamma}(0,\cdot)\|_0^2+|\zeta_{-\gamma}(0,\cdot)|_0^2+\int_{0}^{t}|\zeta_{-\gamma}(s,\cdot)|_0\cdot\D^2(s)ds,$$
which satisfies
$$\begin{aligned}
\varUpsilon'(t)=|\zeta_{-\gamma}(t,\cdot)|_0\cdot\D^2(t)&\leq \frac{1}{2}\left(|\zeta_{-\gamma}(t,\cdot)|_0^2\cdot\D^2(t)+\D^2(t)\right)\\&\leq \frac{C}{2}\varUpsilon(t)\D^2(t)+\frac{1}{2}\D^2(t).
\end{aligned}
$$
By the Gronwall inequality and the a-priori esitmate \eqref{a-prior-est}, it holds
\begin{equation}\label{q5}
|\zeta_{-\gamma}(t,\cdot)|^2_0\leq C\left(\E^2(0)+\F^2(0)+\|v_{-\gamma}(0,\cdot)\|_0^2+|\zeta_{-\gamma}(0,\cdot)|_0^2\right).
\end{equation}

\par By the interpolation inequality, we have
\begin{equation*}
|\zeta(t,\cdot)|_0\lesssim |\zeta_{-\gamma}(t,\cdot)|_0^{\frac{1}{1+\gamma}}\cdot |\nn\zeta(t,\cdot)|_0^{\frac{\gamma}{1+\gamma}},
\end{equation*}
and combining this with \eqref{q5} together, we get
\begin{equation}\label{q6}
|\zeta(t,\cdot)|_0\leq C|\nn\zeta(t,\cdot)|_0^{\frac{\gamma}{1+\gamma}},
\end{equation}
where the positive constant $C$ only relies on the initial data, independent of the time $t$.
By \eqref{q6}, we immediately get
$$\E(t)\lesssim \D^{\frac{\gamma}{1+\gamma}}(t),$$
and then combining it with the a-priori esitmates \eqref{ETF} together, we obtain
\begin{equation*}
\frac{d}{dt}\E^2(t)+\vartheta\E^{\frac{2(1+\gamma)}{\gamma}}(t)\leq 0
\end{equation*}
for some positive constant $\vartheta$.
Therefore we obtain the decay rate
$$\E(t)\leq C(1+t)^{-\frac{\gamma}{2}},$$
where the constant $C>0$ relies on the initial data and $\gamma$, independent of the time $t$.
\end{proof}

\section{Appendix}
\subsection{The local existence}
\par For the completeness of this paper, we show the local existence to equations \eqref{p1} and \eqref{p2}, which is proved in our recent paper \cite{local1, local}.
\begin{proposition}\label{local existence}
Assume that the initial data $(v_0,\zeta_0)\in H^4(\Omega)\times H^{\frac{9}{2}}(\Gamma)$ in the horizontal periodic domain or horizontal infinite domain and $\sup\limits_{x'\in\Gamma}|\zeta_0(x')|<b$. Suppose the condition \eqref{fbass} and the compatibility \eqref{compatibility} hold. Then there exists a positive constant $T$, only depending on the initial data $(v_0,\zeta_0)$, such that equations \eqref{p3}-\eqref{IC} have an unique local strong solution $(v,\zeta)\in L^{\infty}([0,T],H^4(\Omega))\times L^{\infty}([0,T],H^{\frac{9}{2}}(\Gamma))\cap L^{2}([0,T],H^5(\Omega))\times L^{2}([0,T],H^{\frac{9}{2}}(\Gamma))$, satisfying
	$$\sup\limits_{0\leq t\leq T}\left(\E^2(t)+\F^2(t)\right)+\int_0^T(\D^2(t)+\F^2(t))dt\leq C(\E(0)+\F(0))$$
for some positive constant $C$.
\end{proposition}

\begin{remark}
	The proof of Proposition \ref{local existence} is showed by the simple tool, contraction mapping principle. The essential idea is to construct the approximate solution via the new derived wave equation \eqref{nfe} with strong damping term, seeing the details in our recent paper \cite{local1,local}.
\end{remark}

\subsection{Analytic tools}
In the following, we will list some useful analytic tools which are applied to estimate the nonlinear terms.
\begin{lemma}\label{SME}
The following results hold on the smooth domain $\Gamma\subset\R^n$.
\begin{enumerate}
	\item Let $0\leq r\leq s_1\leq s_2$ be such that $s_1>\frac{n}{2}$. Let $f\in H^{s_1}(\Gamma)$ and $g\in H^{s_2}(\Gamma)$. Then $f\cdot g\in H^r(\Gamma)$ and
	$$\|f\cdot g\|_{H^r}\lesssim \|f\|_{H^{s_1}}\cdot\|g\|_{H^{s_2}}. $$
	\item Let $0\leq r\leq s_1\leq s_2$ be such that $s_2>r+\frac{n}{2}$. Let $f\in H^{s_1}(\Gamma)$ and $g\in H^{s_2}(\Gamma)$. Then $f\cdot g\in H^r(\Gamma)$ and
	$$\|f\cdot g\|_{H^r}\lesssim \|f\|_{H^{s_1}}\cdot\|g\|_{H^{s_2}}. $$
	\item Let $0\leq r\leq s_1\leq s_2$ be such that $s_2>r+\frac{n}{2}$. Let $f\in H^{-r}(\Gamma)$ and $g\in H^{s_2}(\Gamma)$. Then $f\cdot g\in H^{-s_1}(\Gamma)$ and
	$$\|f\cdot g\|_{H^{-s_1}}\lesssim \|f\|_{H^{-r}}\cdot\|g\|_{H^{s_2}}. $$
\end{enumerate}
\end{lemma}

\par The important inequalities are described in the next two lemmas to be applied in the negative-index Sobolev space, which can be proved by Hardy-Littlewood-Sobolev inequality, seeing the details in \cite{GuoY1} and the referrences therein.
\begin{lemma}\label{LN1}
Let $\gamma\in(0,1)$. Then
\begin{itemize}[(i)]
	\item If $f\in L^2(\Omega)$, $g$ and $\nn g\in H^1(\Omega)$, we have
	$$\left\||\nn|^{-\gamma}(f\cdot g)\right\|_0\lesssim \|f\|_0\cdot\|g\|_1^{\gamma}\cdot\|\nn g\|_1^{1-\gamma}.$$
	\item If $f\in L^2(\Gamma)$ and $g\in H^1(\Gamma)$, we have
	$$\left||\nn|^{-\gamma}(f\cdot g)\right|_0\lesssim |f|_0\cdot|g|_0^{\gamma}\cdot|\nn g|_0^{1-\gamma}.$$
\end{itemize}
\end{lemma}
\begin{lemma}\label{LN2}
Let $\gamma\in(0,1)$. If $f\in H^k(\Omega)$ for any $k\geq1$, then
$$\left\||\nn|^{-\gamma}\nn^kf\right\|_0\lesssim \|\nn^{k-1}f\|_0^{\gamma}\cdot \|\nn^kf\|_0^{1-\gamma}.$$
\end{lemma}

\section*{Acknowledgements}
The authors would like to thank the anonymous referees for their valuable comments and many useful suggestions which helped to improve the exposition of the current paper.
The authors would like to thank Professor Zhouping Xin for helpful suggestions and discussions.
The research of Hai-Liang Li's work was supported partially by the National Natural Science Foundation of China
(No. 11931010, 12226326, 12226327), by the key research project of Academy
for Multidisciplinary Studies, Capital Normal University, and by the Capacity
Building for Sci-Tech Innovation-Fundamental Scientific Research Funds (No.007/20530290068). Liang's work is partially supported by the National Science Foundation of China
(No. 11701053) and the Fundamental Research Funds for the Central Universities(No.0903005203477, 2020CDJQY-A040 and 2020CDJQY-Z001).

\bibliographystyle{amsplain}

\end{document}